\documentclass{amsart}
\usepackage{bbm, color, mathrsfs, enumitem}
\usepackage{fourier}
\usepackage[foot]{amsaddr}

\def\E{{\mathbb{E}}}
\def\N{{\mathbb{N}}}

\def\P{{\mathbb{P}}}
\def\Q{{\mathbb{Q}}}
\def\R{{\mathbb{R}}}
\def\Z{{\mathbb{Z}}}

\def\Dcal{{\mathcal{D}}}
\def\Ecal{{\mathcal{E}}}
\def\Fcal{{\mathcal{F}}}
\def\Lcal{{\mathcal{L}}}

\def\Ncal{{\mathcal{N}}}

\def\PX{{\mathbf{P}}}
\def\EX{{\mathbf{E}}}
\def\PQ{{\mathbf{Q}}}
\def\PZ{{\mathbf{Z}}}

\newcommand{\indicator}[1]{\mathbbm{1}_{\{#1\}}}
\newcommand{\norm}[1]{\lVert #1 \rVert}

\setenumerate[1]{label=(\roman{enumi})} 

\newtheorem{theorem}{Theorem}[section]
\newtheorem{prop}[theorem]{Proposition}
\newtheorem{lemma}[theorem]{Lemma}
\newtheorem{corollary}[theorem]{Corollary}
\newtheorem*{conv-assumption}{Convergence of L\'evy Processes}
\newtheorem*{tightness-assumption}{Tightness Assumption}

\theoremstyle{remark}

\title[Convergence of local time processes]{Asymptotic behavior of local times of compound Poisson processes with drift in the infinite variance case}

\author[A.\ Lambert]{Amaury Lambert$^1$}
\email{amaury.lambert@upmc.fr}
\address{$^1$Laboratoire de Probabilit\'es et Mod\`eles Al\'eatoires\\
UMR 7599 CNRS and UPMC Univ Paris 06\\
Case courrier 188\\
4 Place Jussieu\\
F-75252 Paris Cedex 05, France}

\author[F.\ Simatos]{Florian Simatos$^2$}
\email{f.simatos@tue.nl}
\address{$^2$Department of Mathematics \& Computer Science \\
Eindhoven University of Technology \\ P.O. Box 513 \\
5600 MB Eindhoven, The Netherlands}

\thanks{The research of Amaury Lambert is funded by project `MANEGE' 09-BLAN-0215 from ANR (French national research agency). While most of this research was carried out, Florian Simatos was affiliated with CWI and sponsored by an NWO-VIDI grant.}

\date{\today}

\newcounter{step}
\newcommand{\newstep}{\noindent \textit{Step~$\thestep$. \addtocounter{step}{1}}}

\begin{document}

\maketitle

\begin{abstract}
	Consider compound Poisson processes with negative drift and no negative jumps, which converge to some spectrally positive L\'evy process with non-zero L\'evy measure. In this paper we study the asymptotic behavior of the local time process, in the spatial variable, of these processes killed at two different random times: either at the time of the first visit of the L\'evy process to $0$, in which case we prove results at the excursion level under suitable conditionings; or at the time when the local time at~$0$ exceeds some fixed level. We prove that finite-dimensional distributions converge under general assumptions, even if the limiting process is not c\`adl\`ag. Making an assumption on the distribution of the jumps of the compound Poisson processes, we strengthen this to get weak convergence. Our assumption allows for the limiting process to be a stable L\'evy process with drift.
	
	These results have implications on branching processes and in queueing theory, namely, on the scaling limit of binary, homogeneous Crump--Mode--Jagers processes and on the scaling limit of the Processor-Sharing queue length process.
\end{abstract}

\bigskip\bigskip\bigskip

\hrule
\vspace{-2mm}
\tableofcontents
\vspace{-8mm}
\hrule

\newpage

\section{Introduction}

Let $X_n$ be a sequence of spectrally positive compound Poisson processes with drift which converges weakly to $X$, then necessarily a spectrally positive L\'evy process. The limiting process has continuous paths as soon as $\E(X_n(1)^2) \to \beta$ for some finite $\beta$, a case henceforth referred to as the finite variance case. In this paper, we focus on the infinite variance 
case, when the limiting L\'evy process has non zero L\'evy measure. Let $L_n$ and $L$ be the local time processes associated to $X_n$ and $X$, respectively, defined by the occupation density formula
\[ \int_0^t f(X_n(s)) ds = \int_{-\infty}^{+\infty} f(x) L_n(x,t) dx \ \text{ and } \ \int_0^t f(X(s)) ds = \int_{-\infty}^{+\infty} f(x) L(x,t) dx. \]

Since $X_n$ converges to $X$ it is natural to also expect $L_n$ to converge to $L$. Note however that the map that to a function associates its local time process is not continuous, and so such a result does not automatically follow from the continuous mapping theorem. In the finite variance case, i.e., when $X$ is Brownian motion, this question has been looked at in Lambert et al.~\cite{Lambert11:0} under the assumption that both $X_n$ and $X$ drift to~$-\infty$. Previously, Khoshnevisan~\cite{Khoshnevisan93:0} investigated this question under different assumptions on $X_n$ and $X$ but with a different goal, namely to derive convergence rates. The goal of the present paper is to investigate the asymptotic behavior of $L_n$ in the infinite variance case, i.e., when the L\'evy measure of $X$ is non-zero.

Except for the two aforementioned papers, it seems that this question has not received much attention. In sharp contrast, there is a rich literature in the closely related case where $X$, still a L\'evy process, is approximated by a sequence $X_n$ of random walks. There are results looking at, e.g., strong and weak invariance principles, convergence rates and laws of the iterated logarithm. Nonetheless, the compound Poisson case in which we will be interested is of practical interest since it finds applications in the theory of branching processes and in queueing theory (see discussion below and Section~\ref{sec:implications}); besides, the setup that we consider offers specific technical difficulties that do not seem to have been addressed in the random walk case. An overview of existing results in the random walk case can give insight into the specific technical difficulties that arise in our framework.

\subsection*{The random walk case}

The most studied case in the random walk case is when $X_n$ is of the form $X_n(t) = S(nt) / n^{1/2}$ with $S$ a lattice random walk with finite variance, say with step size distribution $\xi$, so that $X$ is of the form $\sigma^2 B$ with $B$ a standard Brownian motion. One of the earliest work is in this area was done by Knight~\cite{Knight63:0}, see also~\cite{Borodin81:0, Csaki83:0, Kesten65:0, Perkins82:0, Revesz81:0, Revesz81:1, Stone63:0} for weak convergence results, laws of the iterated logarithm, strong invariance principles and explicit convergence rates. The introduction of Cs\"{o}rg\H{o} and R\'ev\'esz~\cite{Csorgo85:0} presents a good overview of the literature.

When one drops the finite variance assumption on $S$ (but keeps the lattice assumption), $X_n$ is of the form $X_n(t) = S(nt) / s_n$ for some normalizing sequence $(s_n)$ and $X$ is a stable L\'evy process. In this case, significantly fewer results seem available: Borodin~\cite{Borodin84:0} has established weak convergence results, Jain and Pruitt~\cite{Jain84:1} a functional law of the iterated logarithm and Kang and Wee~\cite{Kang97:0} $L_2$-convergence results.

Focusing specifically on weak convergence results, the best results have been obtained by Borodin~\cite{Borodin81:0, Borodin84:0}, who proved that $L_n$ converges weakly to $L$ if $\E(\xi^2) < +\infty$ (finite variance case) or if $\xi$ is in the domain of attraction of a stable law with index $1 < \alpha < 2$ (infinite variance case).
\\

On the other hand, the picture is far to be as complete in the non-lattice case, even when one only focuses on weak convergence results. First of all, in this case the very definition of the local time process is subject to discussion, since in contrast with the lattice case, it cannot be defined by keeping track of the number of visits to different points in space. In Cs\"{o}rg\H{o} and R\'ev\'esz~\cite{Csorgo85:0} for instance, five different definitions are discussed. In the finite variance case, Perkins~\cite{Perkins82:0} has proved that $L_n$ converges to $L$, in the sense of finite-dimensional distributions if $\E(\xi^2) < +\infty$, and weakly if $\E(\xi^4) < +\infty$ and $\limsup_{|t| \to \infty} |\E(e^{it\xi})| < 1$; see also~\cite{Borodin86:0, Csorgo85:0}. In view of the sharp results obtained by Borodin~\cite{Borodin81:0} in the lattice case, it is not clear that the conditions derived by Perkins~\cite{Perkins82:0} to get weak convergence are optimal. Also, note that this discrepancy, in terms of existing results, between the lattice and non-lattice case, reflects the fact that tightness is significantly more difficult in the non-lattice case. In the non-lattice case, the most involved part of the proof concerns the control of small oscillations of the local time, a difficulty that does not appear in the lattice case, as soon as the amplitude of oscillations is smaller than the lattice mesh (see the discussion after Proposition~\ref{prop:case>}).

We finally stress that to our knowledge, the present work is the first study of the asymptotic behavior of $L_n$ in the non-lattice and infinite variance case.

\subsection*{Main results}

In the present paper, we will be interested in $X_n$ of the form $X_n(t) = Y_n(nt) / s_n$ with $(s_n)$ some normalizing sequence, $Y_n(t) = P_n(t) - t$ and $P_n$ a compound Poisson process whose jump distribution $\xi_n \geq 0$ has infinite second moment. We assume that $X_n$ does not drift to $+\infty$ and that it converges weakly to a spectrally positive L\'evy process $X$. We will focus on the variations in space of the local time processes and consider the asymptotic behavior of the processes $L_n(\, \cdot \,, \tau_n)$ for some specific choices of $\tau_n$. Since $L_n(a,t)$ is increasing in $t$ this contains the most challenging part of the analysis of local time processes; moreover, this allows for results at the excursion level (see Theorem~\ref{thm:main-1}). This setup presents two main differences with previous works on random walks.

First, the sequence $X_n$ stems from a \emph{sequence} of compound Poisson processes, when all the aforementioned works in the random walk case consider one random walk $S$ that is scaled. Besides being of practical interest for branching processes and queueing theory, since this allows $X$ to have a drift and, more generally, not to be stable, this variation triggers one important technical difference. Indeed, most of the works on random walks heavily exploit embedding techniques, typically embedding $S$ into $X$. It is therefore not clear whether such techniques could be adapted to a triangular scheme such as the one considered here.

Second, the image set $\{ X_n(t), t \geq 0 \}$ is not lattice and $X_n$ has infinite variance; thus, the corresponding random walk counterpart would be the case of non-lattice random walk with infinite variance which, as mentioned previously, has not been treated. Similarly as Perkins~\cite{Perkins82:0} in the case of non-lattice random walk with finite variance, we will show that finite-dimensional distributions converge under minimal assumptions and that tightness holds under more stringent ones. However, in contrast with Perkins~\cite{Perkins82:0} our assumptions to get tightness will not be in terms of finiteness of some moments but in terms of the specific distribution of $\xi_n$. In particular, under our assumptions the limiting process $X$ can be any process of the form $X(t) = Y(t) - dt$ with $Y$ a spectrally positive stable L\'evy process with index $1 < \alpha < 2$ and $d \geq 0$.

\subsection*{Implications}

As alluded to above, our results have implications for branching processes and in queueing theory, see Section~\ref{sec:implications} for more details. In short, the process $(L_n(a, \tau_n), a \geq 0)$ for the random times $\tau_n$ that will be considered has been shown in Lambert~\cite{Lambert10:0} to be equal in distribution to a (rescaled) binary, homogeneous Crump-Mode-Jagers (CMJ) branching process. Although the scaling limits of Galton-Watson processes and of Markovian CMJ have been exhaustively studied, see~\cite{Grimvall74:0, Helland78:0, Lamperti67:0}, except for Lambert et al.~\cite{Lambert11:0} and Sagitov~\cite{Sagitov94:0, Sagitov95:0} little seems to be known for more general CMJ processes. In particular, we study here for the first time a sequence of CMJ processes that converges towards a non-Markovian limit process.

Also, CMJ processes are in one-to-one relation with busy cycles of the Processor-Sharing queue via a random time change sometimes called Lamperti transformation in the branching literature. Thus our results also show that busy cycles of the Processor-Sharing queue converge weakly to excursions that can be explicitly characterized. Leveraging on general results by Lambert and Simatos~\cite{Lambert12:0}, this implies uniqueness (but not existence) of possible accumulation points of the sequence of queue length processes. This constitutes therefore a major step towards determining the scaling limit (called heavy traffic limit in the queueing literature) of the Processor-Sharing queue in the infinite variance, which has been a long-standing open question.

\subsection*{Organization of the paper}

Section~\ref{sec:notation} sets up the framework of the paper, introduces notation, states assumptions enforced throughout the paper and the two main results. Section~\ref{sec:preliminary} is devoted to some preliminary results on L\'evy processes. In Section~\ref{sec:fd-conv} we prove that under general assumptions, finite-dimensional distributions converge while tightness is taken care of under specific technical assumptions in Section~\ref{sec:tightness}. The (long and tedious) appendix is the most technical part of the paper: it proves the tightness of an auxiliary sequence of processes, which is exploited in Section~\ref{sec:tightness} to prove tightness of the processes of interest.

\subsection*{Acknowledgements} F.\ Simatos would like to thank Bert Zwart for initiating this project and pointing out the reference~\cite{Kella05:0}.

\section{Notation and main results} \label{sec:notation}

\subsection{Space $\Dcal$}

Let $\Dcal$ be the set of functions $f: [0,\infty) \to \R$ which are right-continuous and have a left limit denoted by $f(t-)$ for any $t > 0$. If $f$ is increasing we write $f(\infty) = \lim_{x \to +\infty} f(x) \in [0,\infty]$. For $f \in \Dcal$ we define $\underline f \in \Dcal$ the function $f$ reflected above its past infimum through the following formula:
\[ \underline f(t) = f(t) - \min\left(0, \inf_{0 \leq s \leq t} f(s) \right). \]

For $f_n, f \in \Dcal$ we note $f_n \to f$ for convergence in the Skorohod $J_1$ topology (see for instance Billingsley~\cite{Billingsley99:0} or Chapter~VI in Jacod and Shiryaev~\cite{Jacod03:0}). For any function $f \in \Dcal$, we introduce the family of mappings $(T_f(A, k), A \subset \R, k \geq 0)$ defined recursively for any subset $A \subset \R$ by $T_f(A,0) = 0$ and for $k \geq 1$,
\[ T_f(A, k) = \inf\left\{ t > T_f(A, k-1): f(t) \in A \text{ or } f(t-) \in A \right\}. \]
We will write for simplicity $T_f(A) = T_f(A,1)$ and when $A = \{a\}$ is a singleton, we will write $T_f(a, k)$ and $T_f(a)$ in place of $T_f(A, k)$ and $T_f(A)$, respectively. A function $f \in \Dcal$ is called an excursion if $T_f(0) = +\infty$, or if $T_f(0) \in (0,+\infty)$ and $f(t) = 0$ for all $t \geq T_f(0)$. By $\Ecal \subset \Dcal$ we will denote the set of excursions.

We use the canonical notation for c\`adl\`ag stochastic processes. Let $\Omega = \Dcal$ and $X = (X(t), t \geq 0)$ be the coordinate process, defined by $X(t) = X_\omega(t) = \omega(t)$. We will systematically omit the argument of functional operators when they are applied at $X$; $T(A,k)$ for instance stands for the random variable $T_X(A,k)$. Finally, let $X^0 = X(\, \cdot \, \wedge T(0))$ be the process $X$ stopped upon its first visit to $0$, with $X^0 = X$ when $T(0) = +\infty$.

\subsection{Sequence of L\'evy processes}

For $n \geq 1$, fix $\kappa_n \in (0, +\infty)$ and $\Lambda_n$ some positive random variable. For $x \in \R$, let $\P_n^x$ be the law of a L\'evy process started at $x$ with Laplace exponent $\psi_n$ given by
\[ \psi_n(\lambda) = \lambda - \kappa_n \E \left( 1 - e^{-\lambda \Lambda_n} \right), \ \lambda \geq 0. \]
Noting $\pi_n(da) = \kappa_n \P(\Lambda_n \in da)$, one sees that $X$ under $\P_n^0$ is of the form $P_n(t) - t$ with $P_n$ a compound Poisson process with L\'evy measure $\pi_n$. If we denote by $\eta_n$ the largest root of the convex mapping $\psi_n$, then $\psi_n$ is increasing on $[\eta_n,+\infty)$ and its inverse is denoted by $\phi_n$. In particular, $\phi_n(0) = \eta_n$, which equals zero as soon as $\psi_n'(0+)\ge 0$. 

Let $\Lambda_n^*$ be the forward recurrence time of $\Lambda_n$, also called size-biased distribution, which is the random variable with density $\P(\Lambda_n \geq x) / \E(\Lambda_n)$ with respect to Lebesgue measure. Let $\P_n$ and $\P_n^*$ be the measures defined by $\P_n(\,\cdot\,) = \int \P_n^x(\,\cdot\,) \P(\Lambda_n \in dx)$ and $\P_n^*(\,\cdot\,) = \int \P_n^x(\,\cdot\,) \P(\Lambda_n^* \in dx)$. We will use repeatedly the following result, which characterizes the law of the overshoot of~$X$ under $\P_n^0$ when $X$ under $\P_n^0$ does not drift to $+\infty$ (an assumption that will be enforced throughout the paper). 

\begin{lemma}[Theorem~VII.$17$ in Bertoin~\cite{Bertoin96:0}] \label{lemma:overshoot}
	If $X$ under $\P_n^0$ does not drift to $+\infty$, then $X\big( T((0,\infty)) \big)$ under $\P_n^0(\, \cdot \, | \, T((0,\infty)) < +\infty)$ is equal in distribution to $\Lambda_n^*$.
\end{lemma}

Let $w_n$ be the scale function of $X$ under $\P_n^0$, which is the only absolutely continuous increasing function with Laplace transform
\[ \int_0^\infty e^{-\lambda x} w_n(x) dx = \frac{1}{\psi_n(\lambda)}, \ \lambda > \phi_n(0). \]

It is well-known, and can be easily computed, that $w_n(0) = \lim_{\lambda \to +\infty} (\lambda / \psi_n(\lambda)) = 1$. Scale functions play a central role with regards to exit problems, see forthcoming formula~\eqref{eq:scale}. We now define the sequence of rescaled processes that will be the main focus of the paper.
\\

Fix from now on some sequence $(s_n)$ of strictly positive real numbers, which increases to infinity, and for $n \geq 1$ define $r_n = n / s_n$. Let $\PX_n^{x}$, $\PX_n$ and $\PX_n^*$ be the laws of $X(nt) / s_n$ under $\P_n^{x s_n}$, $\P_n$ and $\P_n^*$, respectively, and let $\underline \PX_n^0$ be the law of $\underline X$ under $\PX_n^0$. Then $\PX_n^a$ is the law of a L\'evy process started at $a$, with L\'evy exponent $\Psi_n(\lambda) = n \psi_n(\lambda / s_n)$, L\'evy measure $\Pi_n(da) = n \kappa_n \P(\Lambda_n / s_n \in da)$ and scale function $W_n(a) = w_n(a s_n) / r_n$. Set also $\Phi_n(\lambda) = s_n \phi_n(\lambda/n)$ so that $\Phi_n(0)$ is the largest root of $\Psi_n$ and $\Phi_n$ is the inverse of $\Psi_n$ on $[\Phi_n(0),+\infty)$. Throughout the paper, we use $\Rightarrow$ to denote weak convergence.

\begin{conv-assumption}
	In the rest of the paper, we consider $\PX^0$ the law of a spectrally positive L\'evy process with infinite variation and non-zero L\'evy measure started at $0$. It is assumed throughout that $(1)$ for each $n \geq 1$, $X$ under $\P_n^0$ does not drift to $+\infty$ and $(2)$ $\PX_n^0 \Rightarrow \PX^0$.
\end{conv-assumption}

We also define $\Psi$ the L\'evy exponent and $W$ the scale function associated to $\PX^0$, as well as $\underline \PX^0$ the law of $\underline X$ under $\PX^0$. The previous assumptions have two immediate consequences: $(1)$ $X$ under $\PX^0$ does not drift to $+\infty$; in particular, $\Psi$ is increasing and letting $\Phi$ be its inverse, it is not hard to show that $\Phi_n \to \Phi$ and $(2)$ $\kappa_n \E(\Lambda_n) \leq 1$ and $\kappa_n \E(\Lambda_n) \to 1$; in particular $\P_n^0$ is close to the law of a critical L\'evy process.
\\

As alluded to above, scale functions play a central role with regards to exit problems. This comes from the following relation, that holds for any $0 \leq a < b \leq \infty$, see for instance Theorem~VII.$8$ in Bertoin~\cite{Bertoin96:0}: 
\begin{equation} \label{eq:scale}
	\PX_n^0 \left( T(-(b-a)) < T(a) \right) = \frac{W_n(a)}{W_n(b)} \ \text{ and } \ \PX^0 \left( T(-(b-a)) < T(a) \right) = \frac{W(a)}{W(b)}.
\end{equation}

\subsection{Local times and excursion measures}

For a given function $f \in \Dcal$, let $\mu_{t,f}$ for any $t \geq 0$ be its occupation measure defined via
\[ \int_\R \varphi(a) \mu_{t,f}(da) = \int_0^t \varphi(f(s)) ds \]
for every measurable function $\varphi \geq 0$. When $\mu_{t,f}$ is absolutely continuous with respect to Lebesgue measure, we denote by $L_f(\, \cdot \, , t)$ its Radon-Nikodym derivative restricted to $[0,\infty)$, satisfying the so-called occupation density formula
\[
	\int_0^t \varphi(f(s)) ds = \int_0^\infty \varphi(a) L_f(a, t) da
\]
for any $t \geq 0$ and any measurable function $\varphi \geq 0$ with $\varphi(a) = 0$ for $a < 0$. The functional $L_f(\, \cdot \, , \,\cdot\,)$ is known as the local time process of $f$ and is uniquely determined up to sets of zero Lebesgue measure (in the space variable). Let $T_f^L(\zeta)$ for $\zeta \geq 0$ be the first time the amount of local time accumulated at level $0$ exceeds $\zeta$:
\[ T_f^L(\zeta) = \inf \left \{ t \geq 0: L_f(0, t) \geq \zeta \right \}. \]

In the sequel we will be interested in the local time process $L = L_X$ under various measures. Under $\PX_n^0$, $L$ is simply given by
\[ L(a,t) = \frac{1}{r_n} \sum_{0\le s\le t} \indicator{X(s) =a} , \ t, a \geq 0, \ \PX_n^0\text{-almost surely}. \]

Further, it is known that $L$ under $\PX^0$ is almost surely well-defined, see for instance Bertoin~\cite{Bertoin96:0}. We will consider $\Ncal$ the excursion measure of $X$ under $\PX^0$ associated to the local time $(L(0, t), t \geq 0)$, and $\underline \Ncal$ the excursion measure of $\underline X$ under $\PX^0$ normalized by considering the local time at $0$ of $\underline X$ equal to $(\min(0,\inf_{[0,t]} X), t \geq 0)$. Under $\Ncal$ and $\underline \Ncal$ we will consider the process $L^0$ defined as follows:
\[ L^0 = \big(L(a, T(0)), a \geq 0 \big). \]

This process indeed is well-defined on $(0, \infty)$ (its value at~$0$ is zero) under $\Ncal$ and $\underline \Ncal$, since $\Ncal$ and $\underline{\Ncal}$ both have the same semigroup in $(0,\infty)$ as the L\'evy process under $(\PX^a, a > 0)$ killed upon reaching zero. Similarly, the process $L_{\underline X}$ is well-defined under $\PX^0$, because on $[\varepsilon,\infty)$ it can be expressed as a finite sum of the local time processes of independent excursions distributed as $\underline{\Ncal}(\, \cdot \mid \sup X > \varepsilon)$. Recall from the beginning of the paragraph that the normalization at~$0$ of $L_{\underline{X}}$ is slightly different, so that the occupation formula for this local time only holds on $(0,\infty)$.

\subsection{Additional assumption for tightness}

It will be shown that the mere assumption $\PX_n^0 \Rightarrow \PX^0$ implies that the finite-dimensional distributions of $L$ under $\PX_n^0$ converge towards those of $L$ under $\PX^0$ (see forthcoming Theorems~\ref{thm:main-1} and~\ref{thm:main-2} for precise statements). However, it is known that the local time process of a L\'evy process is either jointly continuous, or has a very wild behavior, see Barlow~\cite{Barlow88:0} for a general criterion. In the latter case, for every $t, \varepsilon > 0$ the set $\{ L(a,t), a \in \Q \cap (-\varepsilon, \varepsilon) \}$ is almost surely dense in $[0,\infty)$. When focusing on spectrally positive L\'evy processes with infinite variation, Barlow's criterion, which is in general difficult to determine, takes the following simple form.

\begin{lemma} \label{lemma:cond}
	The local time process of $X$ under $\PX^0$ is jointly continuous if and only if
	\begin{equation} \label{eq:cond}
		\int^\infty \frac{d\lambda}{\Psi(\lambda) \sqrt{\log \lambda}} < +\infty.
	\end{equation}
\end{lemma}

\begin{proof}
	It will be proved in Lemma~\ref{lemma:conv-hitting-times} that $\EX^0 (L(0, T(-a)) = W(a)$ for any $a > 0$. Then Lemma~V.$11$ and Theorem~V.$15$ in Bertoin~\cite{Bertoin96:0} ensure that $L$ under $\PX^0$ is jointly continuous if and only if
	\[ \int_0 \sqrt{\log(1/W^{-1}(x))} dx < +\infty. \]
	
	Using the change of variables $y = W^{-1}(x)$ and integrating by parts, it can be shown that
	\[ \int_0^{W(1)} \sqrt{\log (1/W^{-1}(x))} dx = \frac{1}{2} \int_0^1 \frac{W(u)}{u \sqrt{\log(1/u)}} du. \]
	
	Since there exist two finite constants $0 < c \leq C$ such that $c / (u \Psi(1/u)) \leq W(u) \leq C / (u \Psi(1/u))$ for every $u \geq 0$ (see Proposition~III.$1$ or the proof of Proposition~VII.$10$ in Bertoin~\cite{Bertoin96:0}), we obtain
	\[ \int_0 \sqrt{\log (1/W^{-1}(x))} dx < +\infty \Longleftrightarrow \int_0 \frac{du}{u^2 \Psi(1/u) \sqrt{\log(1/u)}} < +\infty \]
	and the change of variables $\lambda = 1/u$ in the last integral therefore gives the result.
\end{proof}

In particular, when~\eqref{eq:cond} fails, $L$ under $\PX^0$ is not even c\`adl\`ag (in the spatial variable) and so cannot be the weak limit of any sequence, when using Skorohod's topology. It is tempting to think that~\eqref{eq:cond} is enough for $L$ under $\PX_n^0$ to converge weakly towards $L$ under $\PX^0$, and we leave this question open for future research. In the present paper we will prove tightness (and hence weak convergence) under the following assumption.

\begin{tightness-assumption}
	In the rest of the paper we fix some $1 < \alpha < 2$ and denote by $\Lambda$ the random variable with tail distribution function $\P(\Lambda \geq s) = (1+s)^{-\alpha}$. We will say that the tightness assumption holds if for $n \geq 1$ we have $s_n = n^{1/\alpha}$ and $\Lambda_n = \Lambda$.
\end{tightness-assumption}

Note that under this assumption, $\PX^0$ is the law of a L\'evy process of the form $Y(t) - dt$ with $Y$ a stable L\'evy process with index $\alpha$ and $d \geq 0$. It is then not difficult to check that~\eqref{eq:cond} is satisfied and so our limiting processes will be continuous. However, we will show weak convergence without the \emph{a priori} knowledge given to us by Lemma~\ref{lemma:cond} that the limiting process is continuous. But since our pre-limit processes make deterministic jumps of size $1/r_n \to 0$, it follows from our approach that $(L(a, T), a \geq 0)$ is continuous for some specific random times $T$, thus proving directly (without resorting to Barlow's more general and complete result) that the local time process of $X$ is continuous.

\subsection{Main results}

The following two theorems are the main results of the paper.

\begin{theorem} \label{thm:main-1}
	For any $a_0 > 0$, the two sequences of processes $L^0$ under $\PX_n^*(\, \cdot \, | \, T(a_0) < T(0))$ and $\PX_n(\, \cdot \, | \, T(a_0) < T(0))$ converge in the sense of finite-dimensional distributions to $L^0$ under $\Ncal(\, \cdot \, | \, T(a_0) < T(0))$ and $\underline \Ncal(\, \cdot \, | \, T(a_0) < T(0))$, respectively. If in addition the tightness assumption holds, then both convergences hold in the sense of weak convergence.
\end{theorem}

When reading the following theorem it is useful to have in mind that $r_n \to +\infty$, since $r_n = 1/W_n(0)$ and it will be proved in Lemma~\ref{lemma:conv-W} that $W_n(0) \to W(0) = 0$.

\begin{theorem} \label{thm:main-2}
	Let $\zeta > 0$ and $(z_n)$ be any integer sequence such that $\zeta_n = z_n / r_n \to \zeta$. Then the sequence of processes $L(\, \cdot \,, T^L(\zeta_n))$ under $\PX_n^0(\, \cdot \, | \, T^L(\zeta_n) < +\infty)$ converges in the sense of finite-dimensional distributions to $L(\, \cdot \,, T^L(\zeta))$ under $\PX^0(\, \cdot \, | \, T^L(\zeta) < +\infty)$. If in addition the tightness assumption holds, then the convergence holds in the sense of weak convergence.
\end{theorem}

Note that since $X$ under $\PX^0$ is assumed to have discontinuous paths, Theorem~$1.1$ in Eisenbaum and Kaspi~\cite{Eisenbaum93:0} guarantees that the limiting process in Theorem~\ref{thm:main-2} is not Markovian. Decomposing the path of $X$ into its excursions away from $0$, it could be also be shown that $L^0$ under $\Ncal$ and $\underline \Ncal$ does not satisfy the Markov property.

As a last remark, we stress that it is possible to extend the proof of Theorem~\ref{thm:main-2} to get convergence of the processes $(L(a,T^L(\zeta_n)), a \in \R)$ instead of $(L(a,T^L(\zeta_n)), a \geq 0)$. Both the proofs of convergence of the finite-dimensional distributions and of tightness can be adapted to this case with no major changes, though at the expense of more computations.

\section{Preliminary results} \label{sec:preliminary}

We prove in this section preliminary results that will be used several times in the paper. We first need results concerning the continuity of hitting times, cf.\ Jacod and Shiryaev~\cite[Proposition VI.2.11]{Jacod03:0} for closely related results.

\begin{lemma} \label{lemma:continuity-hitting-times}
	Let $f_n, f \in \Dcal$ such that $f_n \to f$ and $A$ be any finite subset of $\R$. Assume that $f_n$ for each $n \geq 1$ has no negative jumps and that:
	\begin{enumerate}
		\item \label{i} $f(0) \notin A$;
		\item $T_f(A)$ is finite;
		\item $f$ has no negative jumps;
		\item \label{iv} for any $a \in A$ and $\varepsilon > 0$, $\sup_{[T_f(a), T_f(a) + \varepsilon]} f > a$ and $\inf_{[T_f(a), T_f(a) + \varepsilon]} f < a$;
	\end{enumerate}
	Then $T_{f_n}(A) \to T_{f}(A)$. In particular, if in addition to~\ref{i}--\ref{iv} above $f$ also satisfies the following condition:
	\begin{enumerate}\setcounter{enumi}{4}
		\item \label{v} $f$ is continuous at $T_f(a)$ for each $a \in A$;
	\end{enumerate}
	then $f_n(T_{f_n}(A)) \to f(T_{f}(A))$, with $f_n(T_{f_n}(A))$ being well-defined for $n$ large enough.
\end{lemma}

\begin{proof}
	In the rest of the proof note $A = \{a_k, 1 \leq k \leq K\}$, $T = T_f(A)$ and $T_n = T_{f_n}(A)$. Assume that the result holds for $K = 1$: then for each $a \in A$ one has $T_{f_n}(a) \to T_{f}(a)$. Since $T = \min_{a \in A} T_f(a)$ and $T_n = \min_{a \in A} T_{f_n}(a)$, one deduces that $T_n \to T$. Thus the result only needs to be proved for $K = 1$, which we assume from now on. We then note for simplicity $a = a_1$.

	We first prove that $\liminf_n T_n \geq T$. Let $\varepsilon > 0$ and $t \in (T - \varepsilon, T)$ be such that $f$ is continuous at $t$: then
	\[ \lim_{n \to +\infty} \inf_{[0,t]} |f_n-a| = \inf_{[0,t]} |f-a|. \]

	Since $f(0) \not = a$ and $t < T$, we get $\inf_{[0,t]} |f-a| > 0$ and so $\inf_{[0,t]} |f_n-a| > 0$ for $n$ large enough. For those $n$ we therefore have $T_n \geq t$ and so $\liminf_n T_n \geq t$. Since $t \geq T - \varepsilon$ and $\varepsilon$ is arbitrary, letting $\varepsilon \to 0$ gives $\liminf_n T_n \geq T$.

	We now prove that $\limsup_n T_n \leq T$. Fix $\varepsilon > 0$ and let $t_2 \in (T, T+\varepsilon)$ be a continuity point of $f$ such that $f(t_2) < a$ and $t_1 \in (T, t_2)$ another continuity point of $f$ such that $f(t_1) > a$. Since $t_1$ and $t_2$ are continuity points for $f$, we have $f_n(t_i) \to f(t_i)$ for $i = 1, 2$. Since $f(t_1) > a > f(t_2)$, there exists $n_0 \geq 0$ such that $f_n(t_1) > a > f_n(t_2)$ for all $n \geq n_0$. Since $f_n$ has no negative jumps, for those $n$ there necessarily exists $t_3 \in (t_1, t_2)$ such that $f_n(t_3) = a$ which implies $T_n \leq t_3$. Since $t_3 \leq t_2 \leq T+\varepsilon$ we obtain $T_n \leq T+\varepsilon$ for all $n \geq n_0$ and in particular $\limsup_n T_n \leq T+\varepsilon$. Letting $\varepsilon \to 0$ achieves the proof.
\end{proof}

\begin{lemma}\label{lemma:continuity-hitting-times-Levy}
	For any finite subset $A \subset \R \setminus \{0\}$, $(T(A), X(T(A)))$ under $\PX_n^0(\, \cdot \mid T(A) < +\infty)$ converges weakly to $(T(A), X(T(A)))$ under $\PX^0(\, \cdot \mid T(A) < +\infty)$. If in addition $\min A > 0$, then $(T(A), X(T(A)))$ under $\underline \PX_n^0$ converges weakly to $(T(A), X(T(A)))$ under $\underline \PX^0$. 
\end{lemma}

\begin{proof}
	In the rest of the proof let $H_A \subset \Dcal$ be the set of functions $f \in \Dcal$ satisfying the five conditions~\ref{i}--\ref{v} of Lemma~\ref{lemma:continuity-hitting-times}. Let us first assume that $\min A > 0$ and show the convergence under $\underline \PX_n^0$. By assumption we have $\PX_n^0 \Rightarrow \PX^0$ and so the continuous mapping theorem implies that $\underline \PX_n^0 \Rightarrow \underline \PX^0$. Moreover, since $X$ under $\PX^0$ has by assumption infinite variation, one can easily check that $\underline \PX^0(H_A) = 1$ and so the continuous mapping theorem together with Lemma~\ref{lemma:continuity-hitting-times} give the result for $\underline \PX^0$.

	Let us now show the result under $\PX_n^0(\, \cdot \mid T(A) < +\infty)$, so we do not assume anymore $\min A > 0$. Since $\PX^0(H_A) = 1$, using the same arguments as under $\underline \PX_n^0$, one sees that it is enough to prove that $\PX^0_n(\, \cdot \mid T(A) < +\infty) \Rightarrow \PX^0(\, \cdot \mid T(A) < +\infty)$. If $\min A \leq 0$ or if all the L\'evy processes are critical, then $\PX_n^0(T(A) < +\infty) = \PX^0(T(A) < +\infty) = 1$ and so this last convergence is the same as $\PX_n^0 \Rightarrow \PX^0$. Otherwise, since $\{T(A)<+\infty\} = \{\sup X \ge \min A\}$ it is sufficient to check that $\PX^0 (\sup X= a) =0$ for all $a$ and to prove that $(X, \sup X)$ under $\PX_n^0$ converges to $(X, \sup X)$ under $\PX^0$, which we do now. 

Taking $b=\infty$ in \eqref{eq:scale} shows that
\[	\PX_n^0 \left( \sup X <a  \right) = \frac{W_n(a)}{W_n(\infty)} \ \text{ and } \ \PX^0 \left( \sup X <a  \right) = \frac{W(a)}{W(\infty)},\]
so that indeed $\PX^0 (\sup X= a) =0$ ($W$ is continuous), and the laws of $\sup X$ converge (by Lemma~\ref{lemma:conv-W}).	

As a consequence, the laws of $(X,\sup X)$ are tight. Let $(Y,M)$ be any accumulation point of this sequence. Then $Y$ must be equal in distribution to $X$ under $\PX^0$ and $M$ must be equal in distribution to $\sup X$ under $\PX^0$. As a consequence, $M$ and $\sup Y$ have the same distribution, but $(Y,M)$ does not necessarily have the same law as $(X,\sup X)$. To prove this, it is sufficient to show that $M=\sup Y$. By Skorokhod embedding theorem, we can find a sequence $(Y_n, M_n)$ defined on the same probability space as $(Y,M)$ and converging almost surely to $(Y,M)$, such that for each $n$, $(Y_n,M_n)$ has the law of $(X, \sup X)$ under $\PX^0_n$. Then for any continuity point $t$ of $Y$, $\sup_{[0,t]}Y_n$ converges to $\sup_{[0,t]}Y$. This shows that $M\ge \sup_{[0,t]}Y$, and since continuity points are dense, $M\ge \sup Y$. Now since $M$ and $\sup Y$ have the same distribution, the almost sure inequality $M\ge \sup Y$ actually is an almost sure equality, hence the result.
\end{proof}

\begin{lemma} \label{lemma:conv-hitting-times-2}
	For any $a > 0$,
	\[ \lim_{n \to +\infty}  \kappa_n s_n\PX_n \left( T(a) < T(0) \right) = \underline \Ncal(T(a) < T(0)). \]
\end{lemma}

\begin{proof}
	Let $a > 0$ and $S = \sup_{[0, T(-1)]} \underline X$: the exponential formula for the Poisson point process of excursions gives
	\[  \PX_n^0 \left( S < a \right) = \exp \big(- \kappa_ns_n \PX_n \left( T(a) < T(0) \right) \big) \ \text{ and } \ \PX^0 \left( S < a \right) = \exp \left(- \underline{\Ncal} \left( T(a) < T(0) \right) \right). \]

	On the other hand, using Lemma~\ref{lemma:continuity-hitting-times} and continuity properties of the $\sup$ operator, it is not hard to see that $S$ under $\PX_n^0$ converges weakly to $S$ under $\PX^0$. Since the distribution of $S$ under $\PX^0$ has not atom, we get $\PX_n^0 \left(S < a \right) \to \PX^0 \left(S < a \right)$ which, in view of the last display, concludes the proof.
\end{proof}

It is well-known that scale functions are everywhere differentiable, which justifies the following statement.

\begin{lemma} \label{lemma:conv-W}
	For every $a \geq 0$, it holds that $W_n(a) \to W(a)$ and $W_n'(a) \to W'(a)$. Moreover, $W_n(\infty) \to W(\infty)$.
\end{lemma}

\begin{proof}
First observe that the pointwise convergence of $1/\Psi_n$ to $1/\Psi$, which are the respective Laplace transforms of $W_n$ and $W$, along with Theorem XIII.1.2 in Feller~\cite{Feller71:0} and the continuity of $W$ ensure that $W_n(a) \to W(a)$ for any $a\ge 0$.

On the other hand, it is elementary to show that
	\[ \PX_n \left( T(a) < T(0) \right)=  \frac{W'_n(a)}{\kappa_ns_nW_n(a)} , \]
	and similarly, one can obtain $\underline {\Ncal} (T(a) < T(0)) = W'(a) / W(a)$ as follows (which will prove $W'_n(a) \to W'(a)$ in view of Lemma~\ref{lemma:conv-hitting-times-2}). Setting $h(x):=\underline{\Ncal}(T(x)<T(0))$, by the exponential formula for the Poisson point process of excursions, we get, by \eqref{eq:scale}, for any $0<a<b$,
	\begin{multline*}
	\frac{W(a)}{W(b)}=\PX^0 \left( T(-(b-a)) < T(a) \right)\\
	=\exp\left(-\int_0^{b-a}\underline{\Ncal}(T(a+x)<T(0)) \, dx\right) = \exp\left(-\int_a^{b}h(u) \, du\right).
	\end{multline*}
This entails $h(a)=W'(a)/W(a)$, for example by differentiating the last equality. The last convergence comes from the relations $W_n(\infty) =1/\Psi_n'(0+)$ and $W(\infty) =1/\Psi'(0+)$.
\end{proof}

\begin{lemma} \label{lemma:conv-hitting-times-3}
	For any $\varepsilon > 0$,
	\[ \lim_{n \to +\infty}  \kappa_n s_n\PX_n \left( T(0)>\varepsilon\right) = \underline \Ncal(T(0) > \varepsilon). \]
\end{lemma}
\begin{proof} We abbreviate $T(0)$ into $T$.
According to Chapter~VII in Bertoin~\cite{Bertoin96:0}, under $\PX^0$ (resp.\ $\PX_n^0$), the first passage time process of $X$ in the negative half line is a subordinator with Laplace exponent $\Phi$ (resp.\ $\Phi_n$). This has two consequences.

The first one, obtained by considering $X$ under $\PX^0$, is that $\Phi(\lambda) = d\lambda+\underline{\Ncal}(1-e^{-\lambda T})$, where $d\ge 0$ is a drift coefficient. Actually, $d=0$, since the L\'evy process $X$ under $\PX^0$ has infinite variation, so that $\lim_{\lambda\to\infty} \lambda/\Psi(\lambda) = 0 =  \lim_{\lambda\to\infty} \Phi(\lambda)/\lambda = d$. We thus have  $\Phi(\lambda) = \underline{\Ncal}(1-e^{-\lambda T})$. The second one, obtained by considering $X$ under $\PX_n^0$, is that $\EX_n\left(e^{-\lambda T}\right) = \E \left(e^{-\Phi_n(\lambda) \Lambda_n/s_n}\right)$, so that
\[
\EX_n\left(1-e^{-\lambda T}\right) = \frac{(n/s_n)\Phi_n(\lambda)- \Psi_n(\Phi_n(\lambda))}{n\kappa_n} =\frac{\Phi_n(\lambda)}{\kappa_ns_n}-\frac{\lambda}{n\kappa_n}.
\]

Multiplying each side with $\kappa_n s_n$, letting $n \to +\infty$ and using that $\Phi_n(\lambda) \to \Phi(\lambda)$, this shows that $\kappa_ns_n \EX_n(1-e^{-\lambda T})$ converges to $\underline{\Ncal}(1-e^{-\lambda T})$, i.e.,
\[
\lim_{n\to\infty} \int_0^\infty dx\, e^{-\lambda x}\kappa_ns_n\PX_n (T>x)  =\int_0^\infty dx\, e^{-\lambda x}\underline{\Ncal} (T>x).
\]

Then Theorem XIII.1.2 in Feller~\cite{Feller71:0} implies that $\kappa_ns_n\PX_n (T>x)\to \underline{\Ncal} (T>x)$ for any $x>0$ such that $\underline{\Ncal}( T=x)=0$, that is, for any $x>0$. 
\end{proof}

\begin{lemma} \label{lemma:conv-hitting-times}
	Recall that $r_n = n / s_n$ If $A$ is a finite subset of $\R$ such that $0 \notin A$, then
	\[ \lim_{n \to +\infty} r_n \PX_n^0 \left( T(A) < T(0) \right) = \Ncal(T(A) < T(0)). \]
\end{lemma}

\begin{proof}
	Since $X$ under $\PX_n^0$ is spectrally positive, the following simplifications occur:
	\begin{itemize}
		\item if $\min A > 0$ then $\PX_n^0 \left( T(A) < T(0) \right) = \PX_n^0 \left( T(\min A) < T(0) \right)$;
		\item if $\max A < 0$ then $\PX_n^0 \left( T(A) < T(0) \right) = \PX_n^0 \left( T(\max A) < T(0) \right)$;
		\item if $\min A < 0 < \max A$ then $\PX_n^0 \left( T(A) < T(0) \right) = \PX_n^0 \left( T(\min A) \wedge T(\max A) < T(0) \right)$.
	\end{itemize}

	Thus to prove the result, there are only three cases to consider: $A = \{a\}$ with $a > 0$, $A = \{a\}$ with $a < 0$ or $A = \{ a, b\}$ with $a < 0 < b$.
	\\

	\noindent \textit{First case: $A = \{a\}$ with $a < 0$}. For any $x>0$, $L(x,T(a))$ under $\PX^0$ is~$0$ if $T(a)< T(x)$ and otherwise it is exponentially distributed with parameter $\Ncal(T(a-x) < T(0))$, so we get $v_a(x):=\EX^0 (L(x, T(a))= \PX^0(T(x)<T(a))/\Ncal (T(a-x)<T(0))$. Now since $\PX^0$ only charges L\'evy processes with infinite variation, the map $x \mapsto T(x)$ is $\PX^0$-a.s.\ and $\Ncal$-a.e.\ continuous, so by monotone convergence, $v_a$ is right-continuous and $v_a(0) = 1/\Ncal (T(a)<T(0))$. Now we refer the reader to, e.g., the second display on page~$207$ of Bertoin~\cite{Bertoin97:1}, to check that the L\'evy process with law $\PX^0$ and killed upon hitting $a$ has a continuous potential density, say $u_a$, whose value at~$0$ is equal to $W(-a)$. In addition, by the occupation density formula and Fubini-Tonelli theorem, for any non-negative function $\varphi$ vanishing on the negative half line, $\EX^0 \int_0^{T(a)} \varphi(X(s)) ds =\int_0^\infty \varphi(x) v_a(x) dx$, which, by definition of $u_a$, also equals $\int_0^\infty \varphi(x) u_a(x) dx$. Since $u_a$ is continuous and $v_a$ is right-continuous, $u_a$ and $v_a$ are equal, and in particular $v_a(0) = 1/\Ncal (T(a)<T(0)) = u_a(0) = W(-a)$.

	On the other hand, using the invariance in space of a L\'evy process,~\eqref{eq:scale} and recalling that $r_n = 1/W_n(0)$, we get
	\[ \PX_n^0 \left( T(a) < T(0) \right) = \PX_n^{-a} \left( T(0) < T(-a) \right) = \frac{1}{r_n W_n(-a)}. \]

	The result therefore follows from Lemma~\ref{lemma:conv-W}.
	\\

	\noindent \textit{Second case: $A = \{a, b\}$ with $a < 0 < b$}. Under $\PX_n^0$, the event $\{ T(a) < T(b) \}$ is equal to the event that all excursions away from~$0$ before the first one that hits $a$ (which exists since $X$ does not drift to $+\infty$) do not hit $b$. Hence
	\begin{align*}
		\PX_n^0 \left( T(a) < T(b) \right) & = \sum_{k \geq 0} \PX_n^0 \left( T(a) < T(0) \right) \left\{ \PX_n^0 \left( T(0) < T(a) \wedge T(b) \right) \right\}^k\\
		& = \frac{\PX_n^0 \left( T(a) < T(0) \right)} {\PX_n^0 \left( T(a) \wedge T(b) < T(0) \right)}
	\end{align*}
	so that
	\begin{equation} \label{eq:interm-1}
		\PX_n^0 \left( T(a) \wedge T(b) < T(0) \right) = \frac{\PX_n^0 \left( T(a) < T(0) \right)} {\PX_n^0 \left( T(a) < T(b) \right)} = \frac{W_n(0) / W_n(-a)}{W_n(b) / W_n(b-a)}
	\end{equation}
	using~\eqref{eq:scale} for the last equality. Using the same reasoning, we derive a similar formula for $\Ncal(T(a) \wedge T(b) < T(0))$ as follows. Let $\eta$ be the time of the first atom of the Poisson point process of excursions $(e_t, t \geq 0)$ of $X$ away from~$0$ such that $\inf e_t < a$. Then $\eta$ is distributed like an exponential random variable with parameter $\Ncal(\inf X < a)$ and the Poisson point process $(e_t, t < \eta)$ is independent of $\eta$ and has intensity measure $\Ncal(\, \cdot \, ; \inf X > a)$. Thus by a similar path decomposition as previously,
	\begin{align*}
		\PX^0 \left( T(a) < T(b) \right) & = \E \left\{ \exp \left(-\eta \Ncal(\sup X > b, \inf X > a) \right) \right\}\\
		& = \frac{\Ncal(\inf X < a)} {\Ncal(\inf X < a) + \Ncal(\sup X > b, \inf X > a)}\\
		& = \frac{\Ncal(\inf X < a)} {\Ncal(T(a) \wedge T(b) < T(0))}
	\end{align*}
	from which it follows that
	\begin{equation} \label{eq:interm-2}
		\Ncal(T(a) \wedge T(b) < T(0)) = \frac{\Ncal \left( \inf X < a \right)} {\PX^0 \left( T(a) < T(b) \right)} = \frac{1/W(-a)} {W(b) / W(b-a)}
	\end{equation}
	using $\Ncal(\inf X < a) = 1/W(-a)$ which was proved in the first case. In view of~\eqref{eq:interm-1} and~\eqref{eq:interm-2}, we get
	\[ \lim_{n \to +\infty} r_n \PX_n^0 \left( T(a) \wedge T(b) < T(0) \right) = \Ncal \left( T(a) \wedge T(b) < T(0) \right). \]
	This gives the result.
	\\

	\noindent \textit{Third case: $A = \{a\}$ with $a > 0$}. Remember that $X^0 = X(\, \cdot \, \wedge T(0))$, with $X^0 = X$ in the event $\{T(0) = +\infty\}$.  Consider now $c < 0$: on the one hand, we have by definition
	\[ \PX_n^0 \left( T(c) \wedge T(a) < T(0) \right) = \PX_n^0 \left( \inf X^0 < c \text{ or } \sup X^0 > a \right). \]
	On the other hand,~\eqref{eq:interm-1} for the first equality and~\eqref{eq:scale} for the second one give
	\[
		\PX_n^0 \left( T(c) \wedge T(a) < T(0) \right) = \frac{W_n(0) / W_n(-c)}{W_n(a) / W_n(a-c)} = \frac{\PX_n^a(T(0) < T(a-c))}{r_n W_n(a)}.
	\]

	Because $X$ under $\PX_n^a$ does not drift to $+\infty$, $\PX_n^a(T(0) < T(a-c)) \to 1$ as $c \to -\infty$. Thus letting $c \to -\infty$, we obtain
	\[ \PX_n^0 \left( \inf X^0 = -\infty \text{ or } \sup X^0 > a \right) = \frac{1}{r_n W_n(a)}. \]
	Since under $\PX_n^0$, $\sup X^0 > 0$ implies $\inf X^0 > -\infty$, we have
	\[ \PX_n^0 \left( \inf X^0 = -\infty \text{ or } \sup X^0 > a \right) = \PX_n^0 \left( T(0) = +\infty \right) + \PX_n^0 \left( T(a) < T(0) \right). \]
	Finally, one obtains $\PX_n^0 \left( T(0) = +\infty \right) = 1/(r_n W_n(\infty))$ by taking $a = 0$ and letting $b= +\infty$ in~\eqref{eq:scale}, and in particular
	\[ \PX_n^0 \left( T(a) < T(0) \right) = \frac{1}{r_n W_n(a)} - \frac{1}{r_n W_n(\infty)}. \]

	Similar arguments also imply
	\[ \Ncal \left( T(a) < T(0) \right) = \frac{1}{W(a)} - \frac{1}{W(\infty)} \]
	and so the result follows from Lemma~\ref{lemma:conv-W}.
\end{proof}

\section{Convergence of the finite-dimensional distributions} \label{sec:fd-conv}

Proposition~\ref{prop:conv-fd-shifted-process} establishes the convergence of the finite-dimensional distributions of the local time processes shifted at a positive level, from which we deduce the finite-dimensional convergence of the processes appearing in Theorems~\ref{thm:main-1} and~\ref{thm:main-2} in Corollaries~\ref{cor:conv-fd-P^*} and~\ref{cor:fd-conv-large-initial-condition}.

\begin{prop} \label{prop:conv-fd-shifted-process}
	Let $a_0 > 0$. Then the two sequences of shifted processes $L^0(\, \cdot \, + a_0)$ under $\PX_n^* \left( \, \cdot \mid T(a_0) < T(0) \right)$ and $\PX_n \left( \, \cdot \mid T(a_0) < T(0) \right)$ converge in the sense of finite-dimensional distributions to $L^0(\, \cdot \, + a_0)$ under $\Ncal(\, \cdot \mid T(a_0) < T(0))$ and $\underline \Ncal(\, \cdot \mid T(a_0) < T(0))$, respectively.
\end{prop}

\begin{proof}
	Let $A$ a finite subset of $[a_0,\infty)$ with $a_0 = \min A$, $A_0 = A \cup \{0\}$ and let $(\Q_n, \Q)$ be a pair of probability distributions either equal to
	\[ \left( \PX_n^*(\, \cdot \mid T(a_0) < T(0)), \ \Ncal(\, \cdot \mid T(a_0) < T(0)) \right) \]
	or equal to
	\[ \left( \PX_n(\, \cdot \mid T(a_0) < T(0)), \ \underline \Ncal(\, \cdot \mid T(a_0) < T(0)) \right). \]

	We show that $(L^0(a), a \in A)$ under $\Q_n$ converges weakly to $(L^0(a), a \in A)$ under $\Q$, which will prove the result.
	\\

	\newstep We begin by expressing the laws of $(L^0(a), a \in A)$ under $\Q_n$ and of $(L^0(a), a \in A)$ under $\Q$ in a convenient form, cf.~\eqref{eq:decomposition-L_n} and~\eqref{eq:decomposition-L} below. Let $M = (M_k, k \geq 1)$ be the sequence with values in $A_0$ which keeps track of the successively distinct elements of $A$ visited by $X$ before the first visit to $0$. More specifically, let $\sigma_1 = T(A)$ and for $k \geq 1$ define recursively
	\[ \sigma_{k+1} = \begin{cases}
		\inf \big\{ t > \sigma_k: X(t) \in A_0 \backslash \{X(\sigma_k)\} \big\} & \text{ if } X(\sigma_k) \in A,\\
		\sigma_k & \text{else}.
	\end{cases}  \]

	Note that for every $k \geq 1$, $\sigma_k$ under both $\Q_n$ and $\Q$ is almost surely finite so the above definition makes sense (in this proof we will only work under $\Q_n$ or $\Q$). Let $M_k = X(\sigma_k)$ and for $a \in A$ define
	\[ S(a) = \sum_{k=1}^\infty \indicator{M_k = a}. \]

	For each $a \in A$ this sum is finite, since $M$ under $\Q_n$ and $\Q$ only makes a finite number of visits to $A$ before visiting $0$. When $M_k \in A$, $X$ accumulates some local time at $M_k$ between times $\sigma_k$ and $\sigma_{k+1}$ before visiting $M_{k+1} 
\in A_0 \setminus \{M_k\}$. The amount of local time accumulated depends on whether we are working under $\Q_n$ or $\Q$.

	Under $\Q_n$ and conditionally on $M_k \in A$, $X$ reaches $M_k \in A$ at time $\sigma_k$ and then returns to this point a geometric number of times before visiting $M_{k+1} \in A_0 \setminus \{M_k\}$. Identifying the parameter of the geometric random variables involved, one sees that $(L^0(a), a \in A)$ under $\Q_n$ can be written as follows (with the convention $\sum_1^0 = 0$):
	\begin{equation} \label{eq:decomposition-L_n}
		 (L^0(a), a \in A) = \left( \frac{1}{r_n} \sum_{k=1}^{S(a)} \left( 1 + G_k^n(a) \right), a \in A \right)
	\end{equation}
	where $G_k^n(a)$ is a geometric random variable with success probability $q_n(a)$ given by
	\[ q_n(a) = \PX_n^{a} \left( T \left( A_0 \backslash \{a\} \right) < T(a) \right) \]
	and the random variables $(G_k^n(a), k \geq 1, a \in A)$ are independent and independent of the vector $(S(a), a \in A)$.

	Similarly, under $\Q$ and conditionally on $M_k = a \in A$, it is well-known by excursion theory that $X$ accumulates an amount of local time at level $a$ between times $\sigma_k$ and $\sigma_{k+1}$ which is exponentially distributed with parameter $q(a)$ given by
	\[ q(a) = \Ncal^a(T(A_0 \backslash \{a\}) < T(a)), \]
	where $\Ncal^a$ is the excursion measure of $X$ away from $a$. Iterating this decomposition, one sees that $\left( L^0(a), a \in A \right)$ under $\Q$ can be written as follows:
	\begin{equation} \label{eq:decomposition-L}
		(L^0(a), a \in A) = \left( \sum_{k=1}^{S(a)} E_k(a), a \in A \right)
	\end{equation}
	where $E_k(a)$ is an exponential random variable with parameter $q(a)$ and the random variables $(E_k(a), k \geq 1, a \in A)$ are independent and independent of $(S(a), a \in A)$.

	In view of the decompositions~\eqref{eq:decomposition-L_n} and~\eqref{eq:decomposition-L} and the independence of the random variables appearing in these sums, the result will be proved if we can show that each summand $r_n^{-1} (1+G_k^n(a))$ converges in distribution to $E_k(a)$, and if we can show that the numbers of terms $(S(a),a \in A)$ under $\Q_n$ also converges to $(S(a), a \in A)$ under $\Q$.
	\\

	\newstep We prove that for each $a \in A$, $r_n^{-1} (1+G_1^n(a))$ converges in distribution to $E_1(a)$. This is actually a direct consequence of Lemma~\ref{lemma:conv-hitting-times} which implies that $r_n q_n(a) \to q(a)$ (using the invariance in space of L\'evy processes). The following last step is devoted to proving the convergence of $(S(a), a \in A)$ under $\Q_n$ towards $(S(a), a \in A)$ under $\Q$.
	\\

	\newstep To show that $(S(a), a \in A)$ under $\Q_n$ converges towards $(S(a), a \in A)$ under $\Q$, it is sufficient to show that $M$ under $\Q_n$ converges towards to $M$ under $\Q$. To prove this, we note that both under $\Q_n$ and $\Q$, $M$ is a Markov chain living in the finite state space $A_0$ and absorbed at $0$. Thus to show that $M$ under $\Q_n$ converges to $M$ under $\Q$, it is enough to show that the initial distributions and also the transition probabilities converge.
	\\

	Let us prove the convergence of the initial distributions. We have $\Q_n(M_1 = 0) = \Q(M_1 = 0) = 0$ and for $a \in A$,
	\[ \Q_n(M_1 = a) = \Q_n \left( X(T(A_0)) = a \right) = \Q_n \left( X(T(A)) = a \right). \]

	Similarly,
	\[ \Q(M_1 = a) = \Q \left( X(T(A_0)) = a \right) = \Q \left( X(T(A)) = a \right). \]

	When $\Q_n = \PX_n^*(\, \cdot \mid T(a_0) < T(0))$ and $\Q = \Ncal(\, \cdot \mid T(a_0) < T(0))$, we have
	\[
		\Q_n \left( X(T(A)) = a \right) = \PX_n^* \left( X(T(A)) = a \mid T(a_0) < T(0) \right) = \PX_n^0 \left( X(T(A)) = a \mid T(A) < +\infty \right)
	\]
	and so Lemma~\ref{lemma:continuity-hitting-times-Levy} gives
	\[
		\lim_{n \to +\infty} \Q_n \left( X(T(A)) = a \right) = \PX^0 \left( X(T(A)) = a \mid T(A) < +\infty \right).
	\]
	Since
	\[ \PX^0 \left( X(T(A)) = a \mid T(A) < +\infty \right) = \Ncal \left( X(T(A)) = a \mid T(a_0) < T(0) \right) = \Q\left( X(T(A)) = a \right) \]
	this proves the convergence of the initial distributions in this case. The second case $\Q_n = \PX_n(\, \cdot \mid T(a_0) < T(0))$ and $\Q = \underline \Ncal(\, \cdot \mid T(a_0) < T(0))$ follows similarly by considering the reflected process: we have then
	\[
		\Q_n \left( X(T(A)) = a \right) = \PX_n \left( X(T(A)) = a \mid T(a_0) < T(0) \right) = \underline \PX_n^0 \left( X(T(A)) = a \right)
	\]
	and so Lemma~\ref{lemma:continuity-hitting-times-Levy} gives
	\[
		\lim_{n \to +\infty} \Q_n \left( X(T(A)) = a \right) = \underline \PX^0 \left( X(T(A)) = a \right).
	\]
	Since
	\[ \underline \PX^0 \left( X(T(A)) = a \right) = \underline \Ncal \left( X(T(A)) = a \mid T(a_0) < T(0) \right) = \Q\left( X(T(A)) = a \right) \]
	this proves the convergence of the initial distributions in this case as well.
	\\

	It remains to show that transition probabilities also converge. Note that by definition, in contrast with the initial distributions, transition probabilities of $M$ under $\Q_n$ and $\Q$ do not depend on the case considered. Since $0$ is an absorbing state for $M$ under $\Q_n$ and $\Q$, we only have to show that for any $a \in A$ and $b \in A_0$ with $a \neq b$, we have $\Q_n \left( M_{k+1} = b \mid M_k = a \right) \to \Q \left( M_{k+1} = b \mid M_k = a \right)$. On the one hand,
	\[
		\Q_n \left( M_{k+1} = b \mid M_k = a \right) = \PX_n^{a} \left( X(T(A_0 \backslash \{a\})) = b \right)
	\]
	while on the other hand,
	\[
		\Q \left( M_{k+1} = b \mid M_k = a \right) = \PX^{a} \left( X(T(A_0 \backslash \{a\})) = b \right).
	\]
	Note that $T(A_0 \backslash \{a\})$ is almost surely finite for $a \in A$ both under $\PX_n^a$ and $\PX^a$, so the result follows from Lemma~\ref{lemma:continuity-hitting-times-Levy}.
\end{proof}

\begin{corollary} \label{cor:conv-fd-P^*}
	For any $a_0 > 0$, the two sequences of processes $L^0$ under $\PX_n^* \left( \, \cdot \mid T(a_0) < T(0) \right)$ and $\PX_n \left( \, \cdot \mid T(a_0) < T(0) \right)$ converge in the sense of finite-dimensional distributions to $L^0$ under $\Ncal(\, \cdot \mid T(a_0) < T(0))$ and $\underline \Ncal(\, \cdot \mid T(a_0) < T(0))$, respectively.
\end{corollary}

\begin{proof}
	Let $I \geq 1$, $0 < a_1 < \cdots < a_I$ and $u_i > 0$ for $i = 0, \ldots, I$: we prove the result for the convergence under $\PX_n^*$, the result for $\PX_n$ follows along the same lines, replacing $\PX_n^*$ by $\PX_n$ and $\Ncal$ by $\underline \Ncal$. We show that
	\begin{multline*}
		\lim_{n \to +\infty} \PX_n^* \left( L^0(a_i) \geq u_i, i = 0, \ldots, I \mid T(a_0) < T(0) \right)\\
		= \Ncal \left( L^0(a_i) \geq u_i, i = 0, \ldots, I \mid T(a_0) < T(0) \right).
	\end{multline*}

	If $a_0 \leq a_1$ then this follows directly from Proposition~\ref{prop:conv-fd-shifted-process}. If $a_0 > a_1$ we use Bayes formula:
	\begin{multline*}
		\PX_n^* \left( L^0(a_i) \geq u_i, i = 0, \ldots, I \mid T(a_0) < T(0) \right)\\
		= \frac{\PX_n^* \left( T(a_1) < T(0) \right)}{\PX_n^* \left( T(a_0) < T(0) \right)} \PX_n^* \left( L^0(a_i) \geq u_i, i = 0, \ldots, I \mid T(a_1) < T(0) \right).
	\end{multline*}

	Lemma~\ref{lemma:overshoot} and Lemma~\ref{lemma:conv-hitting-times} (use Lemma~\ref{lemma:conv-hitting-times-2} for $\PX_n$) give
	\[ \lim_{n \to +\infty} \frac{\PX_n^* \left( T(a_1) < T(0) \right)}{\PX_n^* \left( T(a_0) < T(0) \right)} = \frac{\Ncal \left( T(a_1) < T(0) \right)}{\Ncal \left( T(a_0) < T(0) \right)} \]
	and so Proposition~\ref{prop:conv-fd-shifted-process} gives
	\begin{multline*}
		\lim_{n \to +\infty} \PX_n^* \left( L^0(a_i) \geq u_i, i = 0, \ldots, I \mid T(a_0) < T(0) \right)\\
		= \frac{\Ncal(T(a_1) < T(0))}{\Ncal(T(a_0) < T(0))} \Ncal \left( L^0(a_i) \geq u_i, i = 0, \ldots, I \mid T(a_1) < T(0) \right)
	\end{multline*}
	which is equal to $\Ncal (L^0(a_i) \geq u_i, i = 0, \ldots, I \mid T(a_0) < T(0))$. The result is proved.
\end{proof}

\begin{corollary} \label{cor:fd-conv-large-initial-condition}
	Let $\zeta > 0$ and $(z_n)$ be any integer sequence such that $\zeta_n = z_n / r_n \to \zeta$, and recall that $T^L(\zeta) = \inf \left \{ t \geq 0: L(0, t) \geq \zeta \right \}$. Then the sequence of processes $L(\, \cdot \,, T^L(\zeta_n))$ under $\PX_n^0(\, \cdot \mid T^L(\zeta_n) < +\infty)$ converges in the sense of finite-dimensional distributions to $L(\, \cdot \,, T^L(\zeta))$ under $\PX^0(\, \cdot \mid T^L(\zeta) < +\infty)$.
\end{corollary}

\begin{proof}
	Since the marginals at $0$ are deterministic and converge to $\zeta$, we restrict our attention to finite sets $A \subset (0,\infty)$ and we are interested in the weak convergence of the sequence $(L(a, T^L(\zeta_n)), a \in A)$ under $\PX_n^0 ( \, \cdot \mid T^L(\zeta_n) < +\infty )$. Let $a_0 = \min A > 0$: then only those excursions reaching level $a_0$ contribute and so for each $n \geq 1$, $(L(a, T^L(\zeta_n)), a \in A)$ under $\PX_n^0 ( \, \cdot \mid T^L(\zeta_n) < +\infty )$ is equal in distribution to
	\[ \left( \sum_{k=1}^{K_n} L^{a_0}_{n,k}(a), a \in A\right) \]
	where $(L^{a_0}_{n,k}, k \geq 1)$ are i.i.d.\ processes with distribution $L^0$ under $\PX_n^*(\, \cdot \mid T(a_0) < T(0))$, and $K_n$ is an independent random variable distributed as a binomial random variable with parameters $z_n$ and $\PX_n^0 \left( T(a_0) < T(0) \right)$. Since
	\[ \lim_{n \to +\infty} z_n / r_n = \zeta \ \text{ and } \ \lim_{n \to +\infty} r_n \PX_n^0 (T(a_0) < T(0)) = \Ncal(T(a_0) < T(0)), \]
	the second convergence being given by Lemma~\ref{lemma:conv-hitting-times}, the sequence $(K_n)$ converges weakly to a Poisson random variable $K$ with parameter $\zeta \Ncal(T(a_0) < T(0))$. On the other hand, $L^{a_0}_{n,1}$ converges in distribution to $L^0$ under $\Ncal(\, \cdot \mid T(a_0) < T(0))$ by Proposition~\ref{prop:conv-fd-shifted-process}, so the sequence $(L(a, T^L(\zeta_n)), a \in A)$ under $\PX_n^0( \, \cdot \mid T^L(\zeta_n) < +\infty)$ converges weakly to
	\[ \left( \sum_{k=1}^{K} L^{a_0}_{k}(a), a \in A \right) \]
	where $(L^{a_0}_{k}, k \geq 1)$ are independent from $K$ and are i.i.d.\ with common distribution $L^0$ under $\Ncal(\, \cdot \mid T(a_0) < T(0))$. This proves the result.
	Indeed, under $\PX^0(\,\cdot \,| \, T^L(\zeta)<+\infty)$ also, only excursions reaching level $a_0$ contribute to $L(\, \cdot \,, T^L(\zeta))$ on $A$, and there is a Poisson number, say $K'$, with parameter $\zeta \Ncal(T(a_0) < T(0))$, of such excursions. These excursions are i.i.d.\ with common law $\Ncal(\, \cdot \mid T(a_0) < T(0))$, and so their local time processes are i.i.d.\ with common distribution $L^0$ under $\Ncal (\, \cdot \mid T(a_0) < T(0))$. The local time being an additive functional, it is the sum of local times of these $K'$ i.i.d.\ excursions.
\end{proof}

\section{Tightness results} \label{sec:tightness}

Tightness is a delicate issue. In the finite variance case and assuming that the limiting Brownian motion (with drift) drifts to $-\infty$, it follows quickly from a simple queueing argument, see Lambert et al.~\cite{Lambert11:0}. In the infinite variance case of the present paper we prove tightness under the tightness condition stated in Section~\ref{sec:notation}. So in the rest of this section we assume that the tightness condition holds, i.e., for each $n \geq 1$ we have $s_n = n^{1/\alpha}$ and $\Lambda_n = \Lambda$ where $\Lambda$ has tail distribution $\P(\Lambda \geq s) = (1+s)^{-\alpha}$. The main technical result concerning tightness is contained in the following proposition, Appendix~\ref{sec:proof} is devoted to its proof.

\begin{prop} \label{prop:tightness-interm}
	For any $a_0 > 0$, the sequence $L^0(a_0 + \, \cdot \,)$ under $\PX_n^*(\, \cdot \, | \, T(a_0) < T(0))$ converges weakly to the continuous process $L^0(a_0 + \, \cdot \,)$ under $\Ncal(\, \cdot \, | \, T(a_0) < T(0))$.
\end{prop}

We now prove the tightness of the sequences considered in Theorems~\ref{thm:main-1} and~\ref{thm:main-2}. In the sequel, we say that a sequence $(Z_n)$ under $\Q_n$ is C-tight if it is tight and any accumulation point is continuous. It is known, see for instance Chapter~VI in Jacod and Shiryaev~\cite{Jacod03:0}, that this holds if and only if for any $m, \delta > 0$,
\begin{equation} \label{eq:conditions-C-tightness}
	\lim_{a \to +\infty} \limsup_{n \to +\infty} \Q_n \left( \sup_{0 \leq t \leq m} |Z_n(t)| \geq a \right) = 0 \quad \text{ and } \quad \lim_{\varepsilon \to 0} \limsup_{n \to +\infty} \Q_n \left( \varpi_m(Z_n, \varepsilon) \geq \delta \right) = 0
\end{equation}
where $\varpi$ is the following modulus of continuity:
\[ \varpi_m(f, \varepsilon) = \sup \left\{ |f(t) - f(s)| : 0 \leq t, s \leq m, |t-s| \leq \varepsilon \right\}. \]

\begin{lemma}\label{lemma:tightness-sum}
	For each $n \geq 1$, consider on the same probability space an integer valued random variable $K_n$ and a sequence of processes $(Z_{n,k}, k \geq 1)$. Assume that for each $k \geq 1$ the sequence $(Z_{n,k}, n \geq 1)$ is C-tight and that the sequence $(K_n, n \geq 1)$ is tight. Then the sequence $(Z_{n,1} + \cdots + Z_{n, K_n}, n \geq 1)$ is C-tight.
\end{lemma}

\begin{proof}
	Let $S_n = Z_{n,1} + \cdots + Z_{n, K_n}$: we must show that it satisfies~\eqref{eq:conditions-C-tightness}. We show how to control the modulus of continuity, the supremum can be dealt with similar arguments. For any $0 \leq s, t \leq m$ with $|t-s| \leq \varepsilon$, we have
	\[ \left| S_n(t) - S_n(s) \right| \leq \sum_{k=1}^{K_n} \left| Z_{n,k}(t) - Z_{n,k}(s) \right| \leq \sum_{k=1}^{K_n} \varpi_m(Z_{n,k}, \varepsilon) \]
	and so we obtain the following bound, valid for any $\beta > 0$:
	\begin{align*}
		\Q_n \left( \varpi_m(S_n, \varepsilon) \geq \delta \right) & \leq \Q_n \left( \sum_{k=1}^{\beta} \varpi_m(Z_{n,k}, \varepsilon) \geq \delta \right) + \Q_n(K_n \geq \beta)\\
		& \leq \sum_{k=1}^{\beta} \Q_n \left( \varpi_m(Z_{n,k}, \varepsilon) \geq \delta / \beta \right) + \Q_n(K_n \geq \beta).
	\end{align*}
	Since each sequence $(Z_{n,k}, n \geq 1)$ is C-tight this gives
	\[
		\lim_{\varepsilon \to 0} \limsup_{n \to +\infty} \Q_n \left( \varpi_m(S_n, \varepsilon) \geq \delta \right) \leq \limsup_{n \to +\infty} \Q_n(K_n \geq \beta)
	\]
	which goes to $0$ as $\beta$ goes to infinity since $(K_n)$ is tight. The result is proved.
\end{proof}

For $n \geq 1$ and $a > 0$ we denote by $G_n(a)$ a geometric random variable with parameter $1-W_n(0) / W_n(a)$ (remember that $W_n(0) = 1/r_n$), so that according to Lemma~\ref{lemma:conv-W} the sequence $(W_n(0) G_n(a))$ converges in distribution to an exponential random variable with parameter $1/W(a)$. The following proposition, combined with Corollary~\ref{cor:fd-conv-large-initial-condition}, proves Theorem~\ref{thm:main-2}.

\begin{prop} \label{prop:tightness-1}
	Let $\zeta > 0$ and $(z_n)$ be any integer sequence such that $\zeta_n = z_n / r_n \to \zeta$. Then the sequence of processes $L(\, \cdot \,, T^L(\zeta_n))$ under $\PX_n^0(\, \cdot \, | \, T^L(\zeta_n) < +\infty)$ is C-tight.
\end{prop}

\begin{proof}
	Fix any $a_0 > 0$ and let $\tau_n$ be the time of the first visit to~$0$ after time $T(a_0, z_n)$, i.e., $\tau_n = \inf\{ t \geq T(a_0, z_n): X(t) = 0 \}$, and $K_n$ be the number of excursions of $X$ away from~$0$ that reach level $a_0$ before time $\tau_n$. Unless otherwise stated, we work implicitly under $\PX_n^0(\, \cdot \, | \, T(a_0, z_n) < +\infty)$ so that $\tau_n$ and $K_n$ are well-defined. Then the following decomposition holds (using Lemma~\ref{lemma:overshoot}):
	\begin{equation} \label{eq:decomposition}
		L(a_0 + a, \tau_n) = \sum_{k=1}^{K_n} L_{n,k}^{a_0}(a), \ a \geq 0,
	\end{equation}
	where $(L_{n,k}^{a_0}, k \geq 1)$ are independent processes, all equal in distribution to $L^0(a_0 + \, \cdot \, )$ under $\PX_n^*(\, \cdot \, | \, T(a_0) < T(0))$. Note that by Proposition~\ref{prop:tightness-interm}, $(L_{n,1}^{a_0}, n \geq 1)$ is C-tight. In an excursion that reaches level $a_0$, the law of the number of visits of $X$ to $a_0$ is $1 + G_n(a_0)$, thus $K_n$ is equal in distribution to
	\[  \min \left\{ k \geq 1: \sum_{i=1}^k \left( 1 + G_{n,i}(a_0) \right) \geq \zeta_n \right\} \]
	with $(G_{n,i}(a_0), i \geq 1)$ i.i.d.\ random variables with common distribution $G_n(a_0)$. Since $(G_n(a_0) / r_n)$ converges in distribution to an exponential random variable with parameter $1/W(a_0)$ and $z_n / r_n \to \zeta$, the sequence $(K_n)$ converges in distribution to a Poisson random variable with parameter $\zeta / W(a_0)$. In particular, the sequence $(K_n)$ is tight, and combining~\eqref{eq:decomposition}, Lemma~\ref{lemma:tightness-sum}, Proposition~\ref{prop:tightness-interm} and the C-tightness of $(L_{n,1}^{a_0}, n \geq 1)$, one sees that the sequence $L(a_0 + \cdot \,, \tau_n)$ is C-tight. We now prove the desired C-tightness of the sequence $L(\, \cdot \,, T^L(\zeta_n))$ under $\PX_n^0(\, \cdot \, | \, T^L(\zeta_n) < +\infty)$.

	For $a \geq 0$ define $L^{1,a_0}(a)$ and $L^{2,a_0}(a)$ the local time accumulated at level $a_0 + a$ in the time interval $[T(a_0) , T(a_0, z_n)]$ and $[0, T(a_0)) \cup (T(a_0, z_n), \tau_n]$, respectively. Then $L(a_0 + a, \tau_n) = L^{1,a_0}(a) + L^{2,a_0}(a)$ which we write as
	\[ L^{1,a_0}(a) = L(a_0 + a, \tau_n) - L^{2,a_0}(a), \ a \geq 0. \]
	We have just proved that the sequence $L(a_0 + \cdot\,, \tau_n)$ was C-tight, so if we can prove that the sequence $L^{2,a_0}$ is also C-tight, Lemma~\ref{lemma:tightness-sum} will imply the C-tightness of $L^{1,a_0}$. But by invariance in space, $L^{1,a_0}$ under $\PX_n^0(\, \cdot \, | \, T(a_0, z_n) < +\infty)$ is equal in distribution to $L(\,\cdot\,, T^L(\zeta_n))$ under $\PX_n^0(\, \cdot \, | \, T^L(\zeta_n) < +\infty)$ so this will prove the desired result. Hence it remains to prove the C-tightness of $L^{2,a_0}$.

	Let $g(a_0)$ be the left endpoint of the first excursion of $X$ away from $0$ that reaches $a_0$, and define the excursion $Y$ as follows:
	\[ Y(t) = \begin{cases}
		X(t + g(a_0)) & \text{if } 0 \leq t \leq T(a_0) - g(a_0),\\
		X\left((t + T(a_0, z_n) + g(a_0) - T(a_0)) \wedge \tau_n \right) & \text{if } t \geq T(a_0) - g(a_0).
	\end{cases} \]
	Then $Y$ is distributed like $X(\cdot \wedge T(0))$ under $\PX_n^*(\, \cdot \, | \, T(a_0) < T(0))$ and
	\[ L^{3,a_0}(a) \stackrel{\text{def.}}{=} L^{2,a_0}(a) + \frac{1}{r_n} \indicator{a = a_0}, \]
	so that $L^{3,a_0}$ is the local time process above level $a_0$ of the process $Y$. The term $\indicator{a = a_0}$ is here to compensate the fact that $L^{2,a_0}(a)$ is the local time accumulated during the time interval $[0, T(a_0)) \cup (T(a_0, z_n), \tau_n)$ instead of $[0, T(a_0)) \cup [T(a_0, z_n), \tau_n)$, so $L^{2,a_0}$ is missing one visit of $Y$ to $a_0$. Thus $L^{3,a_0}$ is equal in distribution to the process $L^0(a_0 + \, \cdot \,)$ under $\PX_n^*(\, \cdot \, | T(a_0) < T(0))$ and the sequence $L^{3,a_0}$ is therefore C-tight by Proposition~\ref{prop:tightness-interm}. Neglecting the factor $r_n^{-1}$ which vanishes in the limit, this implies that $L^{2,a_0}$ is C-tight and concludes the proof.
\end{proof}

The following proposition, combined with Corollary~\ref{cor:fd-conv-large-initial-condition}, proves Theorem~\ref{thm:main-1}.

\begin{prop} \label{prop:tightness-one}
	For any $a_0 > 0$, the two sequences $L^0$ under $\PX_n^*(\, \cdot \, | \, T(a_0) < T(0))$ and $\PX_n(\, \cdot \, | \, T(a_0) < T(0))$ are C-tight.
\end{prop}

\begin{proof}
	In the rest of the proof fix some $a_0 > 0$. We first show the C-tightness of $L^0$ under $\PX_n^*(\, \cdot \, | \, T(a_0) < T(0))$, from which the C-tightness under $\PX_n(\, \cdot \, | \, T(a_0) < T(0))$ is then derived. Let $z_n = \lceil r_n \rceil$ be the smallest integer larger than $r_n$ and $\zeta_n = z_n / r_n$, so that $\zeta_n \to 1$. Since $T^L(\zeta_n) = T(0,z_n)$, when $T^L(\zeta_n)$ is finite $X$ has at least $z_n$ excursions away from $0$, and so we can define $K_n$ the number of excursions among these $z_n$ first excursions that visit $a_0$. Let $L_n = L(\, \cdot \, , T^L(\zeta_n))$ and $\Q_n = \PX_n^0(\, \cdot \, | \, T^L(\zeta_n) < +\infty, K_n \geq 1)$: using Proposition~\ref{prop:tightness-1} it is easy to show that $L_n$ under $\Q_n$ is C-tight. Indeed, decomposing the path $(X(t), 0 \leq t \leq T(0,z_n))$ into its $z_n$ excursions away from~$0$, one gets
	\[
		\PX_n^0 \left( K_n = 0 \, | \, T^L(\zeta_n) < +\infty \right) = \left\{ \PX_n^0 \left( T(0) < T(a_0) \, | \, T(0) < +\infty \right) \right\}^{z_n}.
	\]

	By duality and~\eqref{eq:scale},
	\[
		\PX_n^0 \left( T(0) < T(a_0) \, | \, T(0) < +\infty \right) = \PX_n^0 \left( T(0) < T(-a_0) \, | \, T(0) < +\infty \right) = \frac{1 - \frac{1}{r_n W_n(a_0)}}{1 - \frac{1}{r_n W_n(\infty)}}
	\]
	which gives $\PX_n^0 \left( K_n = 0 \, | \, T^L(\zeta_n) < +\infty \right) \to e^{-1/W(a_0)}e^{1/W(\infty)} < 1$. It follows in particular that $C = 1 / (\inf_n \PX_n^0 \left( K_n \geq 1 \, | \, T^L(\zeta_n) < +\infty \right))$ is finite, and so for any $m, \varepsilon, \delta > 0$,
	\begin{multline*}
		\Q_n \left( \varpi_m(L_n, \varepsilon) \geq \delta \right) = \frac{\PX_n^0 \left( \varpi_m(L_n, \varepsilon) \geq \delta, K_n \geq 1 \, | \, T^L(\zeta_n) < +\infty \right)}{\PX_n^0 \left( K_n \geq 1 \, | \, T^L(\zeta_n) < +\infty \right)}\\
		\leq C \PX_n^0 \left( \varpi_m(L_n, \varepsilon) \geq \delta \, | \, T^L(\zeta_n) < +\infty \right).
	\end{multline*}

	Since $L_n$ under $\PX_n^0 \left( \, \cdot \, | \, T^L(\zeta_n) < +\infty \right)$ is C-tight by Proposition~\ref{prop:tightness-1}, this implies that the second condition in~\eqref{eq:conditions-C-tightness} holds. One can similarly control the supremum and prove that the first condition in~\eqref{eq:conditions-C-tightness} also holds, which finally proves the C-tightness of $L_n$ under $\Q_n$.

	We now prove the C-tightness of $L^0$ under $\PX_n^*(\, \cdot \, | \, T(a_0) < T(0))$. Let $g(a_0) < d(a_0)$ be the endpoints of the first excursion of $X$ away from $0$ which reaches level $a_0$. In the event $\{T^L(\zeta_n) < +\infty, K_n \geq 1\}$, define $L^1$ as the local time process in the interval $[0, g(a_0)) \cup (d(a_0), T(0,z_n)]$ and $L^2$ the local time process in the interval $[g(a_0), d(a_0)]$. Then under $\Q_n$, the three following properties hold:
	\begin{enumerate}
		\item the process $(L^1(a) + r_n^{-1} \indicator{a = 0}, a \geq 0)$ is equal in distribution to $L(\, \cdot \, , T^L(\zeta_n'))$ under $\PX_n^0(\, \cdot \, | \, T^L(\zeta_n') < +\infty)$, where $\zeta_n' = (z_n-1)/r_n$;
		\item $L^2 = L_n - L^1$;
		\item $L^2$ is equal in distribution to $L^0$ under $\PX_n^*(\, \cdot \mid T(a_0) < T(0))$.
	\end{enumerate}

	Proposition~\ref{prop:tightness-1} and the first property entail that $L^1$ under $\Q_n$ is C-tight. Since $L_n$ under $\Q_n$ has been proved to be C-tight, the second property together with Lemma~\ref{lemma:tightness-sum} imply that $L^2$ under $\Q_n$ is C-tight. The last property finally proves the C-tightness of $L^0$ under $\PX_n^*(\, \cdot \, | \, T(a_0) < T(0))$.

	We now prove that $L^0$ under $\PX_n(\, \cdot \, | \, T(a_0) < T(0))$ is C-tight using the Radon-Nikodym derivative. Let $f(s) = \P(\Lambda \geq s)$: then $\P(\Lambda^* \in ds) / ds = f(s)$ and $\P(\Lambda \in ds) / ds = -f'(s)$ so that $\Lambda$ is absolutely continuous with respect to $\Lambda^*$ with Radon-Nikodym derivative $-f'/f$. In particular, for any measurable function $g$, it holds that
	\[ \E \left( g(\Lambda) \right) = \E \left( -\frac{f'(\Lambda^*)}{f(\Lambda^*)} g(\Lambda^*) \right). \]

	Consequently, since $-f'(s) / f(s) = \alpha / (1+s)$, we get for any $m, \varepsilon$ and $\delta > 0$
	\[ \PX_n \left( \varpi_m(L^0, \varepsilon) \geq \delta \, | \, T(a) < T(0) \right) = \alpha \E \left( \frac{g(\Lambda^*)}{1+\Lambda^*} \right) \leq \alpha \PX_n^* \left( \varpi_m(L^0, \varepsilon) \geq \delta \, | \, T(a) < T(0) \right) \]
	where $g(x) = \PX_n^x \left( \varpi_m(L^0, \varepsilon) \geq \delta \, | \, T(a) < T(0) \right)$. Since $L^0$ under $\PX_n^*(\, \cdot \, | \, T(a_0) < T(0))$ is C-tight by the first part of the proof, we deduce that the second condition in~\eqref{eq:conditions-C-tightness} is satisfied. The supremum can be handled similarly, and therefore the proof is complete.
\end{proof}

\section{Implications for branching processes and queueing theory} \label{sec:implications}

\subsection{Implications for branching processes}

A Crump--Mode--Jagers (CMJ) process, or general branching process, is a stochastic process with non-negative integer values counting the size of a population where individuals give birth to independent copies of themselves; see for instance Haccou et al.~\cite{Haccou07:0} for a definition.

For $n \geq 1$ and $z \in \N$, let $\Z_n^{z}$ (resp.\ $\Z_n^{z*}$) the the law of a binary, homogeneous CMJ branching process with life length distribution $\Lambda_n$ and offspring intensity $\kappa_n$ started with $z$ individuals with i.i.d.\ life lengths with common distribution $\Lambda_n$ (resp.\ $\Lambda_n^*$). Let $\PZ_n^{z}$ (resp.\ $\PZ_n^{z^*}$) be the law of $X(s_n t) / r_n$ under $\Z_n^z$ (resp.~$\Z_n^{z*}$). It follows directly from results in Lambert~\cite{Lambert11:0} that $\PZ_n^1$ is the law of $L^0$ under $\PX_n$ and that $\PZ_n^{z*}$ is the law of $L(\, \cdot \, , T^L(z/r_n))$ under $\PX_n^0(\, \cdot \mid T^L(z/r_n) < +\infty)$. In the sequel, for $\zeta > 0$ introduce $\PZ^{\zeta*}$ the law of $L(\, \cdot \, , T^L(\zeta))$ under $\PX^0(\, \cdot \, | \, T^L(\zeta) < +\infty)$ and $\PZ^*$ (resp.\ $\PZ$) the push-forward of $\Ncal$ (resp.\ $\underline \Ncal$) by $L^0$. In other words, $\PZ^*$ and $\PZ$ are defined by
\[ \PZ^*(A) = \Ncal(L^0 \in A) \text{ and } \PZ(A) = \underline \Ncal(L^0 \in A) \]
for any Borel set $A \subset \Ecal$. The two theorems below are therefore plain reformulations of Theorem~\ref{thm:main-1} and~\ref{thm:main-2}.

\begin{theorem} \label{thm:CMJ-1}
	For any $\varepsilon > 0$, the two sequences $\PZ_n^{1*}(\, \cdot \, | \, T(0) > \varepsilon)$ and $\PZ_n^{1}(\, \cdot \, | \, T(0) > \varepsilon)$ converge in the sense of finite-dimensional distributions to $\PZ^*(\, \cdot \, | \, T(0) > \varepsilon)$ and $\PZ(\, \cdot \, | \, T(0) > \varepsilon)$, respectively. If in addition the tightness assumption holds, then both convergences hold in the sense of weak convergence.
\end{theorem}

\begin{theorem} \label{thm:CMJ-2}
	Let $\zeta > 0$ and $(z_n)$ be any integer sequence such that $z_n / r_n \to \zeta$. Then the sequence $\PZ_n^{z_n*}$ converges in the sense of finite-dimensional distributions to $\PZ^{\zeta*}$. If in addition the tightness assumption holds, then the convergence holds in the sense of weak convergence.
\end{theorem}

Since there is a rich literature on the scaling limits of branching processes, it is interesting to put this result in perspective. CMJ branching processes are (possibly non-Markovian) generalizations of Galton-Watson (GW) branching processes in continuous time. Scaling limits of GW processes have been exhaustively studied since the pioneering work of Lamperti~\cite{Lamperti67:0}, see Grimvall~\cite{Grimvall74:0}. Accumulation points of sequences of renormalized GW processes are called CSBP, they consist of all the continuous-time, continuous state-space, time-homogeneous Markov processes which satisfy the branching property. Via the Lamperti transformation, they are in one-to-one correspondence with spectrally positive L\'evy processes killed upon reaching~$0$, see Lamperti~\cite{Lamperti67:1} or Caballero et al.~\cite{Caballero09:0}.

On the other hand, little is known about scaling limits of CMJ branching processes, except for the Markovian setting where individuals live for an exponential duration and give birth, upon death, to a random number of offspring. Intuitively, in this case the tree representing the CMJ process should not differ significantly from the corresponding genealogical GW tree because the life length distribution has a light tail. And indeed, correctly renormalized, Markovian CMJ processes converge to CSBP, see Helland~\cite{Helland78:0}. The same intuition explains results obtained by Sagitov~\cite{Sagitov94:0, Sagitov95:0}, who proves the convergence of the finite-dimensional distributions of some non-Markovian CMJ processes towards CSBP. It also provides an explanation for the results obtained by Lambert et al.~\cite{Lambert11:0}, where it is proved that binary, homogeneous CMJ branching processes whose life length distribution has a finite variance converge to Feller diffusion, the only CSBP with continuous sample paths.

In the infinite variance case studied in the present paper, and with which Theorems~\ref{thm:CMJ-1} and~\ref{thm:CMJ-2} are concerned, some individuals will intuitively live for a very long time, causing the tree representing the CMJ to differ significantly from the corresponding genealogical GW tree, a difference that should persist in the limit. Our results are consistent with this intuition, since the CMJ studied here converge to non-Markovian processes (see remark following Theorem~\ref{thm:main-2}). To the best of our knowledge, this is the first time that a sequence of branching processes converging to a non-Markovian limit has been studied.

\subsection{Implications for queueing theory}

The Processor-Sharing queue is the single-server queue in which the server splits its service capacity equally among all the users present. 
For $n \geq 1$ and $z \in \N$, let $\Q_n^z$ (resp.\ $\Q_n^{z*}$) be the law of the queue length process of the Processor-Sharing queue with Poisson arrivals at rate $\kappa_n$, service distribution $\Lambda$ and started with $z$ initial customers with i.i.d.\ initial service requirements with common distribution $\Lambda$ (resp.\ $\Lambda^*$). Let $\PQ_n^z$ (resp.\ $\PQ_n^{z*}$) be the law of $X(n t) / n^{1-1/\alpha}$ under $\Q_n^{z}$ (resp.~$\Q_n^{z*}$).
\\

Let $\Ecal_+ \subset \Ecal$ be the set of positive excursions with finite length.  Let $\Lcal: \Ecal_+ \to \Ecal_+$ be the Lamperti transformation: by definition $\Lcal(e)$ for $e \in \Ecal_+$ is the only positive excursion that satisfies $\Lcal(e)(\int_0^t e) = e(t)$ for every $t \geq 0$. In the rest of this section, if $\mu$ is some positive measure on $\Ecal_+$, write $\Lcal(\mu)$ for the push-forward of $\mu$ by $\Lcal$: for any Borel set $A$,
\[ \Lcal(\mu)(A) = \mu(\Lcal \in A). \]

Recall that $X^0 = X(\, \cdot \, \wedge T(0))$: we have the following result, see for instance Chapter~$7.3$ in Robert~\cite{Robert03:0} and references therein.

\begin{lemma} \label{lemma:PS-CMJ}
	For any integer $z \geq 0$ and any $n \geq 1$, $X^0$ under $\PQ_n^{z}$ (resp.\ $\PQ_n^{z*}$) is equal in distribution to $\Lcal(\PZ_n^{z})$ (resp.\ $\Lcal(\PZ_n^{z*})$).
\end{lemma}

In view of Theorems~\ref{thm:CMJ-1},~\ref{thm:CMJ-2} and Lemma~\ref{lemma:PS-CMJ}, it is natural to expect excursions of the queue length process to converge, upon suitable conditioning. The conditioning $\{ T(0) > \varepsilon \}$ of $\PZ^1_n$ as in Theorem~\ref{thm:CMJ-1} is however not convenient in combination with the map $\Lcal$. Instead, it will be more convenient to consider $\PZ_n^1(\, \cdot \, | \, T \circ \Lcal > \varepsilon)$, where from now on we sometimes note $T = T(0)$, so that $T \circ \Lcal (e) = T_{\Lcal(e)}(0)$ for $e \in \Ecal_+$.

From the definition of $\Lcal$, one can see that, for $e \in \Ecal_+$, we have $\Lcal(e)(t) = 0$ if and only if $t \geq \int_0^{T_e(0)} e$, i.e., $T \circ \Lcal(e) = \int_0^{T_e(0)} e = \int_0^{\infty} e$. In particular, when $L^0$ is well-defined we have
\begin{equation} \label{eq:T}
	T \circ \Lcal \circ L^0 = \int_0^\infty L^0 = \int_0^\infty L(a, T) da = T.
\end{equation}

\begin{lemma} \label{lemma:CMJ}
	For any $\varepsilon > 0$, we have $\PZ_n^1 \left( \, \cdot \, | \, T \circ \Lcal > \varepsilon \right) \Rightarrow \PZ \left( \, \cdot \, | \, T \circ \Lcal > \varepsilon \right)$.
\end{lemma}

\begin{proof}
	Starting from Theorem~\ref{thm:CMJ-1} and using Lemma~$4.7$ in Lambert and Simatos~\cite{Lambert12:0}, one sees that it is enough to show that for any $\varepsilon > 0$,
	\begin{equation} \label{eq:tmp-1}
		\lim_{n \to +\infty} \kappa_n s_n \PZ_n^1( T > \varepsilon) = \PZ(T > \varepsilon), \ \lim_{n \to +\infty} \kappa_n s_n \PZ_n^1( T \circ \Lcal > \varepsilon) = \PZ( T \circ \Lcal > \varepsilon)
	\end{equation}
	and
	\begin{equation} \label{eq:tmp-2}
		(X, T \circ \Lcal) \text{ under } \PZ_n^1(\, \cdot \, | \, T > \varepsilon) \text{ converges weakly to } (X, T \circ \Lcal) \text{ under } \PZ(\, \cdot \, | \, T > \varepsilon).
	\end{equation}	

	Since $\PZ_n^1( T > \varepsilon) = \PX_n( T(\varepsilon) < T(0))$ and $\PZ(T > \varepsilon) = \underline \Ncal( T(\varepsilon) < T(0))$, the first limit in~\eqref{eq:tmp-1} is given by Lemma~\ref{lemma:conv-hitting-times-2}. Similarly,~\eqref{eq:T} entails $\PZ_n^1( T \circ \Lcal > \varepsilon) = \PX_n( T > \varepsilon)$ and $\PZ( T \circ \Lcal > \varepsilon) = \underline \Ncal(T > \varepsilon)$ and so the second limit in~\eqref{eq:tmp-1} follows from Lemma~\ref{lemma:conv-hitting-times-3}. Thus to complete the proof it remains to prove~\eqref{eq:tmp-2}.

	Since $T$ under $\PZ_n^1(\, \cdot \, | \, T > \varepsilon)$ is equal in distribution to $\sup X^0$ under $\PX_n(\, \cdot \, | \, \sup X^0 > \varepsilon)$ one can show that it converges weakly to $T$ under $\PZ(\, \cdot \, | \, T > \varepsilon)$. Using similar arguments as in the proof of Lemma~\ref{lemma:continuity-hitting-times-Levy}, one can strengthen this to get the joint convergence of $(X, T)$. Since $T (\Lcal(e)) = \int_0^{T_e(0)} e$ and the map $(f,t) \in \Dcal \times [0,\infty) \mapsto \int_0^t f$ is continuous, the continuous mapping theorem implies that $(X, T \circ \Lcal)$ under $\PZ_n^1(\, \cdot \, | \, T > \varepsilon)$ converges weakly to $(X, T \circ \Lcal)$ under $\PZ(\, \cdot \, | \, T > \varepsilon)$.
\end{proof}

We now state the two main queueing results of the paper. The first one (Theorem~\ref{thm:PS-excursion}) is a direct consequence of Lemma~\ref{lemma:CMJ} together with continuity properties of the map $\Lcal$, it gives results on excursions of the Processor-Sharing queue length process. The second one (Theorem~\ref{thm:PS-process}) leverages on results in Lambert and Simatos~\cite{Lambert12:0} to give two simple conditions under which not only excursions but the full processes converge.

\begin{theorem} \label{thm:PS-excursion}
	Let $\varepsilon, \zeta > 0$ and $(z_n)$ be any integer sequence such that $z_n / r_n \to \zeta$. Then the two sequences of processes $X^0$ under $\PQ_n^{z_n*}$ and $\PQ_n^{1}(\, \cdot \, | \, T > \varepsilon)$ converge weakly to $\Lcal(\PZ^{\zeta*})$ and $\Lcal(\PZ)(\, \cdot \mid T > \varepsilon)$, respectively.
\end{theorem}

\begin{proof}
	Lemma~\ref{lemma:PS-CMJ} shows that $X^0$ under $\PQ^{z_n*}_n$ and $\PQ_n^1(\, \cdot \, | \, T > \varepsilon)$ is equal in distribution to $\Lcal(\PZ_n^{z_n*})$ and $\Lcal(\PZ_n^1(\, \cdot \, | \, T \circ \Lcal > \varepsilon))$, respectively. Then, note that since we consider the Processor-Sharing queue with Poisson arrivals at rate $\kappa_n$ and service distribution $\Lambda$, the tightness assumption holds. Thus $\PZ_n^1(\, \cdot \, | \, T \circ \Lcal > \varepsilon) \Rightarrow \PZ(\, \cdot \, | \, T \circ \Lcal > \varepsilon)$ by Lemma~\ref{lemma:CMJ} and $\PZ_n^{z_n*} \Rightarrow \PZ^{\zeta*}$ by Theorem~\ref{thm:CMJ-2}. Hence the result would be proved if we could show that $\Lcal$ was continuous along the two sequences $\PZ_n^1(\, \cdot \, | \, T \circ \Lcal > \varepsilon)$ and $\PZ_n^{z_n*}$ (using that $\Lcal(\PZ(\, \cdot \mid T \circ \Lcal > \varepsilon)) = \Lcal(\PZ)(\, \cdot \mid T > \varepsilon)$). Lemma~$2.5$ in Lambert et al.~\cite{Lambert11:0} shows that this holds if the two sequences $T$ under $\PZ_n^{z_n*}$ and $\PZ_n^1(\, \cdot \, | \, T \circ \Lcal > \varepsilon)$ are tight; actually, one can show as easily that they converge weakly.

	Indeed, defining $\zeta_n = z_n / r_n$, one sees that $T$ under $\PZ_n^{z_n*}$ is equal in distribution to $\sup_{[0, T^L(\zeta_n)]} X$ under $\PX_n^{0}(\, \cdot \, | \, T^L(\zeta_n) < +\infty)$. One can show using standard arguments that this converges to $\sup_{[0, T^L(\zeta)]} X$ under $\PX^{0}(\, \cdot \, | \, T^L(\zeta) < +\infty)$, which is equal in distribution to  $T$ under $\PZ^{\zeta*}$ and thus shows the weak convergence of $T$ under $\PZ_n^{z_n*}$. Similar arguments for $T$ under $\PZ_n^1(\, \cdot \, | \, T \circ \Lcal > \varepsilon)$ give the result, using also~\eqref{eq:T} in this case.
\end{proof}

\begin{theorem}\label{thm:PS-process}
	If one of the following two conditions is met:
	\begin{itemize}
		\item either the sequence $\PQ_n^{0}$ is tight;
		\item or for any $\eta > 0$,
		\begin{equation} \label{eq:cond-tightness}
			\lim_{\varepsilon \to 0} \limsup_{n \to +\infty} \left[ s_n \PQ_n^1\left( \sup X^0 \geq \eta, T \leq \varepsilon \right) \right] = 0;
		\end{equation}
	\end{itemize}
	then for any $\zeta \geq 0$ and any integer sequence $(z_n)$ such that $z_n / r_n \to \zeta$, the sequence $\PQ_n^{z_n*}$ converges weakly to the unique regenerative process that starts at $\zeta$ and such that its excursion measure is $\Lcal(\PZ)$, its zero set has zero Lebesgue measure and, when $\zeta > 0$, its first excursion is equal in distribution to $\Lcal(\PZ^{\zeta*})$.
\end{theorem}

\begin{proof}
	We check that all the assumptions of Theorem~$3$ in Lambert and Simatos~\cite{Lambert12:0} are satisfied. In order to help the reader, we provide a translation of the notation used in~\cite{Lambert12:0}: the notation $\Ncal$, $\varphi$, $c_n$, $\Ncal_n$, $b_n$, and $v_\infty$ appearing in the statement of~\cite[Theorem~$3$]{Lambert12:0} correspond respectively to (with the notation of the present paper) $\Lcal(\PZ)$, $T$, $s_n \kappa_n$, $\PQ_n^1$, $n \kappa_n$ and $\sup$.

	That $\Lcal(\PZ)(T = +\infty) = \PQ_n^1(T = +\infty) = 0$ follows from the fact that $X$ under $\PX^0$ does not drift to $+\infty$. Let $\varepsilon > 0$: the assumption~$(\textbf{H1})$ in~\cite{Lambert12:0} corresponds to $s_n \kappa_n \PQ_n^1(T > \varepsilon) \to \Lcal(\PZ)(T > \varepsilon)$. Since the workload process associated to $X$ under $\PQ_n^1$ and renormalized in space by $s_n$ is equal in distribution to $\underline X$ under $\PX_n$, we have $\PQ_n^1(T > \varepsilon) = \PX_n(T > \varepsilon)$. Moreover,~\eqref{eq:T} gives $\Lcal(\PZ)(T > \varepsilon) = \underline \Ncal(T > \varepsilon)$ and so $s_n \kappa_n \PQ_n^1(T > \varepsilon) \to \Lcal(\PZ)(T > \varepsilon)$ is just a restatement of Lemma~\ref{lemma:conv-hitting-times-3}. It can be shown similarly that the assumption~$(\textbf{H2})$ in~\cite{Lambert12:0} holds (i.e., for any $\varepsilon, \lambda > 0$ we have $s_n \kappa_n \PQ_n^1( 1 - e^{-\lambda T} ; T \leq \varepsilon) \to \Lcal(\PZ)(1-e^{-\lambda T} ; T \leq \varepsilon)$).

	The convergence of $X^0$ under $\PQ_n^1(\, \cdot \mid T > \varepsilon)$ (which corresponds to $e_\varepsilon(X_n)$ in~\cite{Lambert12:0}) has been taken care of by Theorem~\ref{thm:PS-excursion}. It is easily deduced that $(X^0, T)$ under $\PQ_n^1(\, \cdot \mid T > \varepsilon)$ converges toward $(X^0, T)$ under $\Lcal(\PZ)(\, \cdot \mid T > \varepsilon)$ (which corresponds to $(e_\varepsilon, \varphi \circ e_\varepsilon)(X_n) \Rightarrow (e_\varepsilon, \varphi \circ e_\varepsilon)(X)$ in~\cite{Lambert12:0}).

	Finally, our assumption~\eqref{eq:cond-tightness} corresponds exactly to the last condition~$(6)$ in~\cite{Lambert12:0}, which proves the result if~\eqref{eq:cond-tightness} is assumed. If one assumes tightness of $\PQ_n^0$ instead of~\eqref{eq:cond-tightness}, one can check that the proof of Theorem~$3$ in~\cite{Lambert12:0} goes through, since the assumption~$(6)$ there is only used to show tightness of $\PQ_n^0$. The proof is therefore complete.
\end{proof}

\subsection*{A relation with the height process}

Let $Q$ be the regenerative process, started at $0$, whose zero set has zero Lebesgue measure and with excursion measure $\Lcal(\PZ)$. Duquesne and Le Gall~\cite{Duquesne02:0} have introduced a process $H$, which they call the height process, that codes the genealogy of a CSBP. It is interesting to note that for every $t \geq 0$, $Q(t)$ and $H(t)$ are equal in distribution, as can be seen by combining results from Kella et al.~\cite{Kella05:0} and Limic~\cite{Limic01:0}. Indeed, the first result shows that the one-dimensional distributions of a Processor-Sharing queue and a Last-In-First-Out (LIFO) queue started empty are equal. As for the second result, it shows that in the finite variance case a LIFO queue suitably renormalized converges to $H$ (which is then just a reflected Brownian motion), and the author argues in the remark preceding Theorem~$5$ that this result can be extended to the $\alpha$-stable setting which we have also studied here.

It would be interesting to study whether $Q$ and $H$ share more similarities. In general however, these processes may be dramatically different. For instance, it follows from Lemma~\ref{lemma:cond} and Duquesne and Le Gall~\cite{Duquesne02:0} that if $\Psi$ is such that
\[ \int^\infty \frac{d\lambda}{\Psi(\lambda) \sqrt{\log \lambda}} < +\infty \ \text{ while } \ \int^\infty \frac{d\lambda}{\Psi(\lambda)} = +\infty \]
then $H$ is very wild (not even c\`adl\`ag) while $Q$ is continuous.

\appendix

\section{Proof of Proposition~\ref{prop:tightness-interm}} \label{sec:proof}

In the rest of this section, we fix some $a_0 > 0$ and we assume that the tightness assumption stated in Section~\ref{sec:notation} holds: in particular, $\Lambda_n = \Lambda$ with $\P(\Lambda \geq s) = (1+s)^{-\alpha}$ for some $1 < \alpha < 2$, $n = s_n^{\alpha}$ and $r_n = s_n^{\alpha-1}$. The goal of this section is to prove that the sequence $L^0(a_0 + \, \cdot \,)$ under $\PX_n^*(\, \cdot \, | \, T(a_0) < T(0))$ is tight.
\\

Note that this will prove Proposition~\ref{prop:tightness-interm}: indeed, by Proposition~\ref{prop:conv-fd-shifted-process}, $L^0(a_0 + \, \cdot \,)$ under $\PX_n^*(\, \cdot \, | \, T(a_0) < T(0))$ converges in the sense of finite-dimensional distributions to $L^0(a_0 + \, \cdot \,)$ under $\Ncal(\, \cdot \, | \, T(a_0) < T(0))$. Moreover, the jumps of $L^0(\, \cdot \, + a_0)$ are of deterministic size $1/r_n$. Since $1/r_n \to 0$, any limiting point must be continuous, see for instance Billinglsey~\cite{Billingsley99:0}. Note that this reasoning could therefore be proved to show that $L$ under $\PX^0$ is continuous (in the space variable), a result that is difficult to prove in general (see for instance Barlow~\cite{Barlow88:0}).
\\

Under the tightness assumption, the scale function $w_n$ enjoys the following useful properties. The convexity and smoothness properties constitute one of the main reasons for making the tightness assumption.

\begin{lemma} \label{lemma:prop-w}
	For each $n \geq 1$, $w_n$ is twice continuously differentiable and concave, its derivative $w_n'$ is convex, $w_n(0) = 1$ and $w_n'(0) = \kappa_n$. Moreover,
	\[ \sup \left\{ \frac{w_n(t)}{(1+t)^{\alpha}}: n \geq 1, t > 0 \right\} < +\infty \]
	and finally, there exist $n_0 \geq 1$ and $t_0 > 0$ such that $w_n(t_0) \geq 2$ for all $n \geq n_0$.
\end{lemma}

\begin{proof}
	The smoothness of $w_n$ follows from Theorem~$3$ in Chan et al.~\cite{Chan10:0} since $f(s) = \P(\Lambda \geq s)$ is continuously differentiable with $|f'(0)| < +\infty$. The convexity properties follow from Theorem~$2.1$ in Kyprianou et al.~\cite{Kyprianou10:0} since $f$ is log-convex and $\psi'_n(0) \geq 0$. The formulas for $w_n(0)$ and $w'_n(0)$ are well-known, see for instance Kuznetsov et al.~\cite{Kuznetsov11:0}. We now prove the two last assertions.
	
	First of all, note that
	\[ \sup \left\{ \frac{w_n(t)}{(1+t)^{\alpha}}: n \geq 1, t > 0 \right\} \leq \max \left( \sup \left\{ w_n(1): n \geq 1 \right\}, \sup \left\{ \frac{w_n(t)}{t^{\alpha-1}}: n \geq 1, t \geq 1 \right\} \right). \]
	
	Let $\overline \psi$ be the L\'evy exponent given by $\overline \psi(\lambda) = \lambda - (\alpha-1) \E( 1 - e^{-\lambda \Lambda})$ with corresponding scale function $\overline w$. Since $\kappa_n \to \alpha-1$, $\P_n^0$ converges in distribution to the law of the L\'evy process with L\'evy exponent $\overline \psi$, and so it can be shown similarly as in the proof of Lemma~\ref{lemma:conv-W} that $w_n(1) \to \overline w(1)$. The first term $\sup \left\{ w_n(1): n \geq 1 \right\}$ appearing in the above maximum is therefore finite. As for the second term, since for any $n \geq 1$ we have $\kappa_n \E(\Lambda) = \kappa_n / (\alpha-1) \leq 1$ by assumption, we get $\overline \psi \leq \psi_n$ and by monotonicity it follows that $w_n \leq \overline w$. Moreover, it is known that there exists a finite constant $C > 0$ such that $\overline w(t) \leq C / (t \overline \psi(1/t))$ for all $t > 0$, see Proposition~III.$1$ or the proof of Proposition~VII.$10$ in Bertoin~\cite{Bertoin96:0}. In particular,
	\[ \sup \left\{ \frac{w_n(t)}{t^{\alpha-1}}: n \geq 1, t \geq 1 \right\} \leq \sup \left\{ \frac{\overline w(t)}{t^{\alpha-1}}: t \geq 1 \right \} \leq C \sup \left\{ \frac{t^\alpha}{\overline \psi(t)}: 0 < t \leq 1 \right \}. \]
	Since $\P(\Lambda \geq s) = (1+s)^{-\alpha}$ one can check that there exists some constant $\beta > 0$ such that $\overline \psi(t) \sim \beta t^{\alpha}$ as $t \to 0$, which shows that the last upper bound is finite and proves the desired result.
	
	To prove the last assertion of the lemma, consider $n_0$ large enough such that $\underline \kappa = \inf_{n \geq n_0} \kappa_n > 1/2$ (remember that $\kappa \to 1/(\alpha-1)>1$). Let $\underline \psi$ be the L\'evy exponent given by $\underline \psi(\lambda) = \lambda - \underline \kappa \E(1-e^{-\lambda \Lambda})$ and corresponding scale function $\underline w$. By monotonicity, we get $w_n \geq \underline w$ for any $n \geq n_0$, and one easily checks that $\underline w(\infty) = 1 / (1- \underline \kappa / (\alpha-1))$. Since by choice of $\underline \kappa$ this last limit is strictly larger than $2$, there exists $t_0 > 0$ such that $\underline w(t_0) \geq 2$. This proves the result.
\end{proof}

Since we are interested in limit theorems, we will assume in the sequel without loss of generality that there exists $t_0 > 0$ such that $w_n(t_0) \geq 2$ for all $n \geq 1$, and we henceforth fix such a $t_0$. We first give a short proof of Proposition~\ref{prop:tightness-interm} based on the two following technical results, see Theorem~$13.5$ in Billingsley~\cite{Billingsley99:0}.

\begin{prop} [Case $(b-a) \vee (c-b) \leq t_0 / s_n$] \label{prop:case<}
	For any $A > a_0$, there exist finite constants $C, \gamma \geq 0$ such that for all $n \geq 1$, $\lambda > 0$ and $a_0 \leq a < b < c \leq A$ with $(b-a) \vee (c-b) \leq t_0 / s_n$,
	\[
		\PX_n^* \left( \left| L^0(b) - L^0(a) \right| \wedge \left| L^0(c) - L^0(b) \right| \geq \lambda \, | \, T(a) < T(0) \right) \leq C \frac{(c-a)^{3/2}}{\lambda^\gamma}.
	\]
\end{prop}

\begin{prop} [Case $b-a \geq t_0 / s_n$] \label{prop:case>}
	For any $A > a_0$, there exist finite constants $C, \gamma \geq 0$ such that for all $n \geq 1$, $\lambda > 0$ and $a_0 \leq a < b \leq A$ with $b-a \geq t_0 / s_n$,
	\[
		\PX_n^* \left( \left| L^0(b) - L^0(a) \right| \geq \lambda \, | \, T(a) < T(0) \right) \leq C \frac{(b-a)^{3/2}}{\lambda^\gamma}.
	\]
	Moreover, the constant $\gamma$ can be taken equal to the constant $\gamma$ of Proposition~\ref{prop:case<}.
\end{prop}

At this point, it must be said that the case $(b-a) \vee (c-b) \leq t_0 / s_n$ is much harder than the case $b-a \geq t_0 / s_n$. The reason is that in the former case, the bound $(c-a)^{3/2}$ cannot be achieved without taking the minimum between $|L^0(b) - L^0(a)|$ and $|L^0(c) - L^0(b)|$. Considering only one of these two terms gives a bound which can be shown to decay only linearly in $c-a$, which is not sufficient to establish tightness. This technical problem reflects that, in the well-studied context of random walks, tightness in the non-lattice case is harder than in the lattice one, where typically small oscillations, i.e., precisely when $(b-a) \vee (c-b) \leq t_0 / s_n$, are significantly easier to control.

\begin{proof} [Proof of Proposition~\ref{prop:tightness-interm} based on Propositions~\ref{prop:case<} and~\ref{prop:case>}]
	According to Theorem~$13.5$ in Billingsley~\cite{Billingsley99:0}, it is enough to show that for each $A > a_0$, there exist finite constants $C, \gamma \geq 0$ and $\beta > 1$ such that for all $n \geq 1$, $\lambda > 0$ and $a_0 \leq a < b < c \leq A$,
	\begin{equation} \label{eq:goal-tightness}
		\PX_n^* \left( \left| L^0(b) - L^0(a) \right| \wedge \left| L^0(c) - L^0(b) \right| \geq \lambda \, | \, T(a_0) < T(0) \right) \leq C \frac{(c-a)^\beta}{\lambda^\gamma}.
	\end{equation}
	
	Fix $n \geq 1$, $\lambda > 0$ and $a_0 \leq a < b < c$ and let $E = \{ |L^0(b) - L^0(a)| \wedge |L^0(c) - L^0(b)| \geq \lambda \}$. Since $X$ under $\PX_n^*$ is spectrally positive, we have $E \subset \{ L(a) > 0 \}$ and so
	\[ \PX_n^*(E, T(a_0) < T(0)) = \PX_n^*(E) = \PX_n^*(E, T(a) < T(0)). \]
	
	Thus Bayes formula entails
	\[
		\PX_n^* \left( E \, | \, T(a_0) < T(0) \right) = \frac{\PX_n^* \left( T(a) < T(0) \right)}{\PX_n^* \left( T(a_0) < T(0) \right)} \PX_n^* \left( E \, | \, T(a) < T(0) \right)
	\]
	and since $\PX_n^* \left( T(a) < T(0) \right) \leq \PX_n^* \left( T(a_0) < T(0) \right)$, we get
	\begin{multline*}
		\PX_n^* \left( \left| L^0(b) - L^0(a) \right| \wedge \left| L^0(c) - L^0(b) \right| \geq \lambda \, | \, T(a_0) < T(0) \right)\\
		\leq \PX_n^* \left( \left| L^0(b) - L^0(a) \right| \wedge \left| L^0(c) - L^0(b) \right| \geq \lambda \, | \, T(a) < T(0) \right).
	\end{multline*}
	
	Thus~\eqref{eq:goal-tightness} follows from the previous inequality together with either Proposition~\ref{prop:case<} when $(b-a) \vee (c-b) \leq t_0 / s_n$, or Proposition~\ref{prop:case>} when $b-a \geq t_0 / s_n$. In the last remaining case where $c-b \geq t_0 / s_n$, we derive similarly the following upper bound:
	\begin{multline*}
		\PX_n^* \left( \left| L^0(b) - L^0(a) \right| \wedge \left| L^0(c) - L^0(b) \right| \geq \lambda \, | \, T(a_0) < T(0) \right)\\
		\leq \PX_n^* \left( \left| L^0(c) - L^0(b) \right| \geq \lambda \, | \, T(a_0) < T(0) \right)\\
		\leq \PX_n^* \left( \left| L^0(c) - L^0(b) \right| \geq \lambda \, | \, T(b) < T(0) \right).
	\end{multline*}
	Proposition~\ref{prop:case>} then concludes the proof.
\end{proof}

The rest of this section is devoted to the proof of Propositions~\ref{prop:case<} and~\ref{prop:case>}. Our analysis relies on an explicit expression of the law of $(L^0(b) - L^0(a), L^0(c) - L^0(b))$ under $\PX_n^{x_0}(\, \cdot \, | \, T(a) < T(0))$. For $0 < a < b < c$ and $x_0 > 0$ we define
\[
	p_n^{x_0}(a) = \PX_n^{x_0} \left( T(a) < T(0) \right), \, p_{n,\xi}^{x_0}(a,b) = \PX_n^{x_0} \left( T(b) < T(a) \, | \, T(a) < T(0) \right)
\]
as well as
\[ p_{n,\theta}^{x_0}(a,b,c) = \PX_n^{x_0} \left( T(c) < T(b) \, | \, T(b) < T(a) < T(0) \right). \]
Remember that $G_n(a)$ denotes a geometric random variable with parameter $p_n^a(a)$, and from now on we adopt the convention $\sum_1^{-1} = \sum_1^0 = 0$.

\begin{lemma} \label{lemma:fd}
	For any $0 < a < b < c$ and $x_0  > 0$, the random variable
	\[ \left( r_n L^0(b) - r_n L^0(a), r_n L^0(c) - r_n L^0(b) \right) \]
	under $\PX_n^{x_0}( \, \cdot \, | \, T(a) < T(0))$ is equal in distribution to
	\begin{equation} \label{eq:fd}
		\left( \xi_n^{x_0} + \sum_{k=1}^{G_n(a)} \xi_{n,k}^a, \, \theta_n^{x_0} \indicator{\xi_n^{x_0} \geq 0} + \sum_{k=1}^{N_{n,a}} \theta_{n,k}^a + \sum_{k=1}^{N_{n,b}^{x_0}} \theta_{n,k}^b \right)
	\end{equation}
	where
	\[ N_{n,a} = \sum_{k=1}^{G_n(a)} \indicator{\xi_{n,k}^a \geq 0} \text{ and } N_{n,b}^{x_0} = (\xi_n^{x_0})^+ + \sum_{k=1}^{G_n(a)} (\xi_{n,k}^a)^+. \]
	All the random variables $\xi_n^{x_0}$, $\xi_{n,k}^a$, $\theta_n^{x_0}$, $\theta_{n,k}^a$, $\theta_{n,k}^b$ and $G_n(a)$ are independent. For any $u > 0$ and $k \geq 1$, $\theta_{n,k}^u$ is equal in distribution to $\xi_n^u$ and $\theta_{n,k}^u$ to $\theta_n^u$, where the laws of $\xi_n^u$ and $\theta_n^u$ are described as follows: for any function~$f$,
	\begin{equation} \label{eq:def-xi}
		\E\left[ f(\xi_n^u) \right] = (1-p_{n,\xi}^u(a,b))f(-1) + p_{n,\xi}^u(a,b) \E \left[ f(G_n(b-a)) \right]
	\end{equation}
	and
	\begin{equation} \label{eq:def-theta}
		\E\left[ f(\theta_n^u) \right] = (1-p_{n,\theta}^u(a,b,c))f(-1) + p_{n,\theta}^u(a,b,c) \E \left[ f(G_n(c-b)) \right].
	\end{equation}
\end{lemma}

\begin{proof}
	In the rest of the proof we work under $\PX_n^{x_0}(\, \cdot \, | \, T(a) < T(0))$. By definition, $r_n L^0(a)$ and $r_n L^0(b)$ is the number of visits of $X$ to $a$ and $b$ in $[0, T(0)]$, respectively. Thus if $\beta$ is the number of visits of $X$ to $b$ in $[0, T(a)]$ and $\beta_{k}^{a}$ the number of visits of $X$ to $b$ between the $k$th and $(k+1)$st visit to $a$, we have
	\[ r_n L^0(b) - r_n L^0(a) = (\beta - 1) + \sum_{k=1}^{r_n L^0(a) - 1} (\beta_{k}^a - 1). \]
	
	Decomposing the path of $X$ between successive visits to $a$ and using the strong Markov property, one easily checks that all the random variables of the right hand side are independent and that $r_n L^0(a)$, $\beta_k^a$ and $\beta$ are respectively equal in distribution to $1 + G_n(a)$, $\xi_{n}^{a}+1$ and $\xi_{n}^{x_0}+1$. This shows that the random variable $r_n L^0(b) - r_n L^0(a)$ is equal in distribution to $\xi_n^{x_0} + \sum_{k=1}^{G_n(a)} \xi_{n,k}^a$.
	
	To describe the law of $(r_n L^0(b) - r_n L^0(a), r_n L^0(c) - r_n L^0(b))$ one also needs to count the number of visits of $X$ to $c$: if $X$ visits $b$ before $a$, it $X$ may visit $c$ before the first visit to $b$; it may also visit $c$ each time it goes from $a$ to $b$; finally it may also visit $c$ between two successive visits to $b$. These three different ways of visiting $c$ are respectively taken into account by the terms $\theta_n^{x_0}$, $\theta_{n,k}^a$ and $\theta_{n,k}^b$.
\end{proof}

The previous result readily gives the law of $(r_n L^0(b) - r_n L^0(a), r_n L^0(c) - r_n L^0(b))$ under $\PX_n^*( \, \cdot \, | \, T(a) < T(0))$: this law can be written as
\begin{equation} \label{eq:fd-integrated}
	\left( \widetilde \xi_n^a + \sum_{k=1}^{G_n(a)} \xi_{n,k}^a, \, \widetilde \theta_n^a \indicator{\widetilde \xi_n^a \geq 0} + \sum_{k=1}^{N_{n,a}} \theta_{n,k}^a + \sum_{k=1}^{\widetilde N_{n,b}^a} \theta_{n,k}^b \right)
\end{equation}
where $\widetilde N_{n,b}^a = (\widetilde \xi_n^a)^+ + \sum_{k=1}^{G_n(a)} (\xi_{n,k}^a)^+$ and the random variables $\xi_{n,k}^a$, $\theta_{n,k}^a$, $\theta_{n,k}^b$, $G_n(a)$ and $N_{n,a}$ are as described in Lemma~\ref{lemma:fd}. Moreover, these random variables are also independent from the pair $(\widetilde \xi_n^a, \widetilde \theta_n^a)$ whose distribution is given by
\begin{equation} \label{eq:def-tilde}
	\E \left[ f(\widetilde \xi_n^a, \widetilde \theta_n^a) \right] = \int \E \left[ f(\xi_n^{x_0}, \theta_n^{x_0}) \right] \P(\chi_n^a \in dx_0)
\end{equation}
where from now on $\chi_n^a$ denotes a random variable equal in distribution to $X(0)$ under $\PX_n^*(\, \cdot \, | \, T(a) < T(0))$. For convenience we will sometimes consider that $\chi_n^a$ lives on the same probability space and is independent from all the other random variables. This will for instance allow us to say that the random vector~\eqref{eq:fd-integrated} conditional on $\{ \chi_n^a = x_0 \}$ is equal in distribution to the random vector~\eqref{eq:fd}.
\\

In order to exploit~\eqref{eq:fd-integrated} and prove Propositions~\ref{prop:case<} and~\ref{prop:case>}, we will use a method based on controlling the moments, following similar lines as Borodin~\cite{Borodin81:0, Borodin84:0}. As Propositions~\ref{prop:case<} and~\ref{prop:case>} suggest, we need to distinguish the two cases $(b-a) \vee (c-b) \leq t_0 / s_n$ and $b-a \geq t_0 / s_n$ (remember that $t_0$ is a fixed number such that $w_n(t_0) \geq 2$ for each $n \geq 1$, see the discussion after Lemma~\ref{lemma:prop-w}).
\\

In the sequel, we need to derive numerous upper bounds. The letter $C$ then denotes constants which may change from line to line (and even within one line) but never depend on $n$, $a$, $b$, $c$, $x_0$ or $\lambda$. They may however depend on other variables, such as typically $a_0$, $A$ or $t_0$.
\\

Before starting, let us gather a few relations and properties that will be used repeatedly (and sometimes without comments) in the sequel. First, it stems from~\eqref{eq:scale} that
\[ \PX_n^{x_0}(T(a) < T(0)) = \frac{W_n(a) - W_n(a-x_0)}{W_n(a)} = \frac{\int_0^{x_0} W_n'(a-u) du}{W_n(a)}, \ 0 < x_0 \leq a. \]

Moreover, $(1)$ $\kappa_n \leq 1$; $(2)$ $W'_n(0) = \kappa_n s_n / r_n = \kappa_n s_n^{2-\alpha}$, $W_n \geq 0$ is increasing, $W'_n \geq 0$ is decreasing and $-W_n'' \geq 0$ is decreasing (as a consequence of Lemma~\ref{lemma:prop-w} together with the identity $W_n(a) = w_n(a s_n) / r_n$); $(3)$ for every $a > 0$, the sequences $(W_n(a), n \geq 1)$ and $(W_n'(a), n \geq 1)$ are bounded away from $0$ and infinity (by Lemma~\ref{lemma:conv-W}).

\subsection{Case $(b-a) \vee (c-b) \leq t_0 / s_n$} \label{sub:<} The goal of this subsection is to prove Proposition~\ref{prop:case<} through a series of lemmas.

\begin{lemma}\label{lemma:j}
	For any $A > a_0$, there exists a finite constant $C$ such that for any $n \geq 1$, any $a_0 \leq a \leq A$ and any $0 \leq x_0 \leq y \leq a$,
	\begin{equation} \label{eq:three-bounds}
		j_{n,a}(x_0, y) \leq \begin{cases}
			C x_0 (a-y) s_n^{2-\alpha} & \text{ if } x_0 \geq a/4,\\
			C x_0 (a-y) & \text{ if } x_0 \leq a/4 \text{ and } a/2 \leq y,\\
			C x_0 W_n'(y-x_0) & \text{ if } x_0 \leq a/4 \text{ and } y \leq a/2,
		\end{cases}
	\end{equation}
	where $j_{n,a}(x_0, y) = W_n(a-x_0) W_n(y) - W_n(y-x_0) W_n(a)$.
\end{lemma}

\begin{proof}
	The derivation of these three bounds is based on the following identity:
	\begin{multline*}
		j_{n,a}(x_0, y) = \int_0^{a-y} W_n'(u+y) du \int_0^{x_0} W_n'(u+y-x_0) du\\
		+ W_n(y) \int_0^{x_0} \int_0^{a-y} (-W_n'') (u+v+y-x_0) du dv.
	\end{multline*}
	
	The term $W_n(y)$ appearing in the second term of the right-hand side is upper bounded by the finite constant $\sup_{n \geq 1} W_n(A)$, and so does not play a role as far as~\eqref{eq:three-bounds} is concerned. Assume that $x_0 \geq a/4$: then $y \geq a_0/4$ and so
	\[ \int_0^{a-y} W_n'(u+y) du \int_0^{x_0} W_n'(u+y-x_0) du \leq (a-y) W'_n(a_0/4) x_0 W'_n(0) \leq C (a-y) x_0 s_n^{2-\alpha}. \]
	
	On the other hand,
	\[ \int_0^{x_0} \int_0^{a-y} (-W_n'') (u+v+y-x_0) du dv \leq (a-y) \int_0^{x_0} (-W_n'') \leq (a-y) W'_n(0). \]
	
	This gives the desired upper bound of the form $C x_0 s_n^{2-\alpha}$, since $W'_n(0) \leq s_n^{2-\alpha}$ and $1 \leq C x_0$ when $x_0 \geq a/4$ (writing $1 = x_0 / x_0 \leq 4 x_0 / a_0$). This proves~\eqref{eq:three-bounds} in the case $x_0 \geq a/4$. Assume now that $x_0 \leq a/4 \leq a/2 \leq y$, so that $y \geq a_0/2$ and $y-x_0 \geq a_0/4$. Then
	\[ \int_0^{a-y} W_n'(u+y) du \int_0^{x_0} W_n'(u+y-x_0) du \leq (a-y) W_n'(a_0/2) x_0 W_n'(a_0/4) \leq C x_0(a-y) \]
	and
	\[ \int_0^{x_0} \int_0^{a-y} (-W_n'') (u+v+y-x_0) du dv \leq \int_0^{x_0} \int_0^{a-y} (-W_n'') (u+v+a_0/4) du dv. \]
	
	Since $-W_n''$ is decreasing, so is the function $\varphi(u) = \int_0^{a-y} (-W_n'') (u+v+a_0/4) dv$. Differentiating, this immediately shows that the function $z \mapsto z^{-1} \int_0^z \varphi$ is decreasing and since $x_0 \geq a_0$, this gives
	\[ \int_0^{x_0} \int_0^{a-y} (-W_n'') (u+v+a_0/4) du dv \leq x_0 a_0^{-1} \int_0^{a-y} \int_0^{a_0} (-W_n'') (u+v+a_0/4) du dv. \]
	
	Further, exploiting that $-W_n''$ is decreasing, we obtain
	\begin{align*}
		\int_0^{x_0} \int_0^{a-y} (-W_n'') (u+v+a_0/4) du dv & \leq x_0 a_0^{-1} (a-y) \int_0^{a_0} (-W_n'') (u+a_0/4) du\\
		& \leq C x_0 (a-y).
	\end{align*}
	
	This proves~\eqref{eq:three-bounds} in the case $x_0 \leq a/4 \leq a/2 \leq y$. Assume now that $x_0 \leq a/4$ and that $y \leq a/2$: then
	\[ \int_0^{a-y} W_n'(u+y) du \int_0^{x_0} W_n'(u+y-x_0) du \leq C x_0 W_n'(y-x_0) \]
	and
	\[ \int_0^{x_0} \int_0^{a-y} (-W_n'') (u+v+y-x_0) du dv \leq x_0 \int_0^{a-y} (-W_n'') (u+y-x_0) du \leq x_0 W'_n(y-x_0). \]
	
	This concludes the proof of~\eqref{eq:three-bounds}, 
\end{proof}

\begin{lemma} \label{lemma:p-xi+theta}
	For any $A > a_0$, there exists a finite constant $C$ such that for any $n \geq 1$, any $a_0 \leq a < b \leq A$ and any $x_0 \leq a$,
	\begin{equation} \label{eq:bound-p-xi}
		1-p_{n,\xi}^{x_0}(a,b) \leq C (b-a) s_n
	\end{equation}
	and for any $n \geq 1$, any $a_0 \leq a < b < c \leq A$ with $b-a \leq t_0 / s_n$ and any $x_0 \leq b$,
	\begin{equation} \label{eq:bound-p-theta}
		1 - p_{n, \theta}^{x_0}(a, b, c) \leq C (c-b) s_n.
	\end{equation}
\end{lemma}

\begin{proof}
	Let us first prove~\eqref{eq:bound-p-xi}, so assume until further notice $x_0 \leq a$. Let in the rest of the proof $\tau_a = \inf\{ t \geq 0: X(t) \notin (0,a) \}$: then the inclusion
	\[ \{ T(a) < T(b), T(a) < T(0) \} \subset \{ a \leq X(\tau_a) < b \} \]
	holds $\PX_n^{x_0}$-almost surely and leads to
	\begin{align*}
		1-p_{n,\xi}^{x_0}(a,b) = \frac{\PX_n^{x_0} \left( T(a) < T(b), T(a) < T(0) \right)}{\PX_n^{x_0} \left( T(a) < T(0) \right)} & \leq \PX_n^{x_0}\left( a \leq X(\tau_a) < b \right) \frac{W_n(a)}{\int_0^{x_0} W'_n(a-u) du}\\
		& \leq C x_0^{-1} \PX_n^{x_0}\left( a \leq X(\tau_a) < b \right).
	\end{align*}
	
	Thus~\eqref{eq:bound-p-xi} will be proved if we can show that $\PX_n^{x_0}\left( a \leq X(\tau_a) < b \right) \leq C x_0 (b-a) s_n$. Let $h_n(z) = \alpha \kappa_n s_n^{\alpha+1} (1+zs_n)^{-\alpha-1}$ be the density of the measure $\Pi_n$ with respect to Lebesgue measure; note that it is decreasing and that the sequence $(h_n(z))$ is bounded for any $z > 0$. Corollary~$2$ in Bertoin~\cite{Bertoin97:1} gives
	\begin{equation} \label{eq:density}
		\PX_n^{x_0} \left( X(\tau_a-) \in dy, \Delta X(\tau_a) \in dz \right) = u_{n,a}(x_0, y) h_n(z) dz dy, \ y \leq a \leq y + z,
	\end{equation}
	where $\Delta X(t) = X(t) - X(t-)$ for any $t \geq 0$ and
	\[ u_{n,a}(x_0, y) = \frac{W_n(a-x_0) W_n(y)}{W_n(a)} - W_n(y-x_0) \indicator{y \geq x_0}, \]
	so it follows that
	\begin{multline*}
		\PX_n^{x_0}\left( a \leq X(\tau_a) < b \right) = \int \indicator{y < a, a-y \leq z < b-y} u_{n,a}(x_0, y) h_n(z) dz dy\\
		\leq (b-a) \int_0^a u_{n,a}(x_0, y) h_n(a-y) dy.
	\end{multline*}
	
	Hence~\eqref{eq:bound-p-xi} will be proved if we can show that $\int_0^a u_{n,a}(x_0, y) h_n(a-y) dy \leq C x_0 s_n$. We have by definition of $u_{n,a}$
	\begin{multline*}
		\int_0^a u_{n,a}(x_0, y) h_n(a-y) dy = \frac{W_n(a-x_0)}{W_n(a)} \int_0^{x_0} W_n(y) h_n(a-y) dy\\
		+ \int_{x_0}^a \left( \frac{W_n(a-x_0) W_n(y)}{W_n(a)} - W_n(y-x_0) \right) h_n(a-y) dy
	\end{multline*}
	which readily implies
	\begin{multline} \label{eq:intermediate}
		\int_0^a u_{n,a}(x_0, y) h_n(a-y) dy \leq C W_n(a-x_0) \int_0^{x_0} h_n(a-y) dy\\
		+ C \int_{x_0}^a j_{n,a}(x_0, y) h_n(a-y) dy
	\end{multline}
	with $j_{n,a}(x_0, y) = W_n(a-x_0) W_n(y) - W_n(y-x_0) W_n(a)$ as in Lemma~\ref{lemma:j}. We will show that each term of the above right hand side is upper bounded by a term of the form $C x_0 s_n$. Let us focus on the first term, so we want to show that
	\[
		W_n(a-x_0) \int_0^{x_0} h_n(a-y) dy \leq C x_0 s_n.
	\]
	Assume first that $x_0 \leq a/2$, then $a-y \geq a-x_0 \geq a/2 \geq a_0 / 2$ for any $y \leq x_0$ which gives
	\[
		W_n(a-x_0) \int_0^{x_0} h_n(a-y) dy \leq W_n(A) x_0 h_n(a_0/2) \leq C x_0 \leq C x_0 s_n.
	\]

	Assume now $x_0 \geq a/2$: since $h_n$ is the density of $\Pi_n$, we have
	\[ W_n(a-x_0) \int_0^{x_0} h_n(a-y) dy \leq W_n(a-x_0) \Pi_n((a-x_0, \infty)) = \kappa_n s_n \frac{w_n((a-x_0)s_n)}{(1 + (a-x_0)s_n)^{\alpha}}. \]
	
	Since $1 \leq C x_0$ (because $x_0 \geq a_0/2$), the desired upper bound of the form $C x_0 s_n$ follows from Lemma~\ref{lemma:prop-w}. We now control the second term of the right hand side in~\eqref{eq:intermediate}, i.e., we have to show that
	\[ \int_{x_0}^a j_{n,a}(x_0, y) h_n(a-y) dy \leq C x_0 s_n. \]

	In the case $x_0 \geq a/4$, the first bound in~\eqref{eq:three-bounds} gives
	\begin{multline*}
		\int_{x_0}^a j_{n,a}(x_0, y) h_n(a-y) dy \leq C x_0 s_n^{2-\alpha} \int_{x_0}^a \frac{(a-y) s_n^{\alpha+1}}{(1+(a-y)s_n)^{\alpha+1}} dy\\
		= C x_0 s_n \int_0^{(a-x_0) s_n} \frac{y}{(1+y)^{\alpha+1}} dy \leq C x_0 s_n.
	\end{multline*}
	
	Assume from now on that $x_0 \leq a/4$ and decompose the interval $[x_0, a]$ into the union $[x_0, a/2] \cup [a/2,a]$. For $[a/2, a]$,~\eqref{eq:three-bounds} gives
	\[
		\int_{a/2}^a j_{n,a}(x_0, y) h_n(a-y) dy \leq C x_0 \int_{a/2}^a \frac{(a-y) s_n^{\alpha+1}}{(1+(a-y)s_n)^{\alpha+1}} dy \leq C x_0 s_n^{\alpha-1} \leq C x_0 s_n
	\]
	since $\alpha - 1 < 1$. For $[x_0, a/2]$, \eqref{eq:three-bounds} gives, using $a-y \geq a_0/2$ when $y \leq a/2$,
	\[
		\int_{x_0}^{a/2} j_{n,a}(x_0, y) h_n(a-y) dy \leq C x_0 h_n(a_0/2) \int_{x_0}^a W_n'(y-x_0) dy \leq C x_0 \leq C x_0 s_n.
	\]
	
	This finally concludes the proof of~\eqref{eq:bound-p-xi}, which we use to derive~\eqref{eq:bound-p-theta}. Assume from now on that $x_0 \leq b$, we have by definition
	\begin{multline*}
		1-p_{n,\theta}^{x_0}(a,b,c) = \frac{\PX_n^{x_0} \left( T(b) < T(c), T(b) < T(a) < T(0) \right)}{\PX_n^{x_0} \left( T(b) < T(a) < T(0) \right)}\\
		\leq \frac{\PX_n^{x_0} \left( T(b) < T(c), T(b) < T(0) \right)}{\PX_n^{x_0} \left( T(b) < T(a) < T(0) \right)} = \frac{\PX_n^{x_0} \left( T(b) < T(0) \right)}{\PX_n^{x_0} \left( T(b) < T(a) < T(0) \right)} \left( 1 - p_{n,\xi}^{x_0}(b,c) \right).
	\end{multline*}
	
	Since $\PX_n^z(T(b) < T(0)) = \int_0^z \varphi$ where $\varphi(u) = W'_n(b-u) / W_n(b)$ is increasing, it is readily shown by differentiating that $z \in [0,b] \mapsto z^{-1} \PX_n^z(T(b) < T(0))$ is also increasing. Thus for $x_0 \leq b$ we obtain
	\[ \PX_n^{x_0} \left( T(b) < T(0) \right) \leq x_0 b^{-1} \PX_n^b \left( T(b) < T(0) \right) \leq C x_0. \]
	In combination with~\eqref{eq:bound-p-xi} and the fact that $\PX_n^{x_0}(X(\tau_a) \geq b) \leq \PX_n^{x_0}(T(b) < T(a) < T(0))$, this entails
	\[ 1-p_{n,\theta}^{x_0}(a,b,c) \leq C (c-b) s_n \frac{x_0}{\PX_n^{x_0}\left( X(\tau_a) \geq b \right)}. \]
	Hence~\eqref{eq:bound-p-theta} will be proved if we show that $x_0 \leq C \PX_n^{x_0}\left( X(\tau_a) \geq b \right)$. If follows from~\eqref{eq:density} that
	\[ \PX_n^{x_0} \left( X(\tau_a) \geq b \right) = \int_0^a u_{n,a}(x_0, y) \kappa_n s_n^{\alpha} \P(\Lambda \geq (b-y) s_n) dy \]
	and because $\P(\Lambda \geq u) = (1+s)^{-\alpha}$, it can be checked that $\P(\Lambda \geq u+v) \geq \P(\Lambda \geq u) \P(\Lambda \geq v)$ for any $u, v > 0$, so that
	\[ \PX_n^{x_0} \left( X(\tau_a) \geq b \right) \geq \P(\Lambda \geq (b-a) s_n) \int_0^a u_{n,a}(x_0, y) \kappa_n s_n^{\alpha} \P(\Lambda \geq (a-y) s_n) dy. \]
	
	In view of~\eqref{eq:density}, this last integral is equal to $\PX_n^{x_0}(X(\tau_a) \geq a) = \PX_n^{x_0}(T(a) < T(0))$ so finally, using $(b-a) s_n \leq t_0$ we get
	\[ \frac{x_0}{\PX_n^{x_0} \left( X(\tau_a) \geq b \right)} \leq \frac{x_0 W_n(a)}{\P(\Lambda \geq t_0) (W_n(a) - W_n(a-x_0))} \leq C \]
	which proves~\eqref{eq:bound-p-theta}.
\end{proof}

\begin{lemma} \label{lemma:bound-xi}
	There exists a finite constant $C$ such that for all $a_0 \leq a < b$ and all $n \geq 1$,
	\[ \left| \E \left( \xi_n^a \right) \right| \leq C (b-a) s_n \big / (r_n W_n(a)). \]
\end{lemma}

\begin{proof}
	Starting from $a$, $X$ (under $\PX_n^a$) makes $1+G_n(a)$ visits to $a$. Decomposing the path $(X(t), 0 \leq t \leq T(0))$ between successive visits to $a$, one gets
	\begin{align*}
		\PX_n^{a}\left( T(b) > T(0) \right) & = \E \left[ \left\{ \PX_n^{a}(T(a) < T(b) \, | \, T(a) < T(0)) \right\}^{G_n(a)} \right]\\
		& = \E \left[ (1 - p_{n,\xi}^a(a,b))^{G_n(a)} \right].
	\end{align*}
	By definition, the left hand side is equal to $1-p_n^a(b)$, so integrating on $G_n(a)$ gives
	\[
		\E \left[ (1 - p_{n,\xi}^a(a,b))^{G_n(a)} \right] = \frac{1-p_n^a(a)}{1 - (1 - p_{n,\xi}^a(a,b))p_n^a(a)} = 1 - p_n^a(b)
	\]
	which gives
	\begin{equation} \label{eq:formula-p-xi-a}
		1 - p_{n,\xi}^{a}(a,b) = \frac{W_n(b-a)W_n(a) - W_n(b) W_n(0)}{W_n(b-a)(W_n(a) - W_n(0))}.
	\end{equation}
	Let $p_n = p_n^{b-a}(b-a)$ and $p_{n,\xi} = p_{n,\xi}^a(a,b)$: we have
	\[ \E \left( \xi_n^a \right) = p_{n,\xi} - 1 + p_{n,\xi} \E(G_n(b-a)) = p_{n,\xi} - 1 + p_{n,\xi} \frac{p_n}{1-p_n} = \frac{p_{n,\xi}}{1-p_n}-1. \]
	Plugging in~\eqref{eq:scale} and~\eqref{eq:formula-p-xi-a} gives after some computation
	\[
		\E \left( \xi_n^a \right) = \frac{W_n(b) - W_n(a) - (W_n(b-a) - W_n(0))}{W_n(a) - W_n(0)} = \frac{\int_0^a \int_0^{b-a} W_n''(u+v) du dv}{\int_0^a W_n'}.
	\]
	
	Since $W'_n \geq 0$ and $W_n'' \leq 0$, this gives
	\begin{equation} \label{eq:formula-xi}
		\left| \E \left( \xi_n^a \right) \right| = \frac{\int_0^a \int_0^{b-a} (-W_n'')(u+v) du dv}{\int_0^a W_n'}
	\end{equation}
	and since $W_n'$ is convex, we get
	\[
		\left| \E \left( \xi_n^a \right) \right| \leq (b-a) \frac{\int_0^a (-W_n'')}{\int_0^a W_n'} = \frac{(b-a) W_n'(0)}{W_n(a)} \frac{W_n(a) (W_n'(0) - W_n'(a))}{W_n'(0) (W_n(a) - W_n(0))}.
	\]
	
	Since $W_n'(0) \leq s_n / r_n$ and
	\[
		\frac{W_n(a) (W_n'(0) - W_n'(a))}{W_n'(0) (W_n(a) - W_n(0))} \leq \frac{W_n(a_0)}{W_n(a_0) - W_n(0)} \leq C,
	\]
	the proof is complete.
\end{proof}

To control the higher moments of $\xi_n^a$ and also the moments of the $\theta$'s we introduce the following constants:
\begin{equation} \label{eq:C_i}
	C_i = \sup \left\{ \frac{\E \left( (G_n(\delta))^i \right)}{\delta s_n} : n \geq 1, 0 \leq \delta \leq t_0 / s_n \right\}, \ i \geq 1.
\end{equation}

\begin{lemma} \label{lemma:C_i}
	For any integer $i \geq 1$, the constant $C_i$ is finite.
\end{lemma}

\begin{proof}
	Using the concavity of $w_n$, one gets $w_n(\delta s_n) \leq w_n(0) + \delta s_n w_n'(0) \leq 1 + \delta s_n$ since $w_n(0) = 1$ and $w_n'(0) = \kappa_n \leq 1$ by Lemma~\ref{lemma:prop-w}. Hence
	\[ p_n^\delta(\delta) = \frac{w_n(\delta s_n) - w_n(0)}{w_n(\delta s_n)} \leq \frac{\delta s_n}{1 + \delta s_n} \leq \min \left( \delta s_n, \frac{t_0}{1 + t_0} \right) \]
	where the last inequality holds for $\delta s_n \leq t_0$. In particular, for any $i \geq 1$ we have
	\[ \frac{\E \left( (G_n(\delta))^i \right)}{\delta s_n} = \frac{\E \left( (G_n(\delta))^i \right)}{p_n^\delta(\delta)} \frac{p_n^\delta(\delta)}{\delta s_n} \leq \sup_{0 \leq p \leq \varepsilon} \left( p^{-1} \E \left( G_p^i \right) \right) \]
	where $\varepsilon = t_0 / (1 + t_0) < 1$ and $G_p$ is a geometric random variable with parameter $p$. It is well-known that
	\[ \E(G_p^i) = \frac{p P_{i-1}(p)}{(1-p)^i} \]
	where $P_i$ is the polynomial $P_i(p) = \sum_{k=0}^i T_{k,i} p^i$ with $T_{k,i} \geq 1$ the Eulerian numbers. Since $P_i$ satisfies $P_i(0) = 1$, one easily sees that for any $\varepsilon < 1$,
	\[ \sup_{0 \leq p \leq \varepsilon} \left( p^{-1} \E \left( G_p^i \right) \right) < +\infty \]
	which achieves the proof.
\end{proof}

\begin{lemma}
	For any $A > a_0$ and $i \geq 1$, there exists a finite constant $C$ such that for all $n \geq 1$ and all $a_0 \leq a < b < c \leq A$ with $(b-a) \vee (c-b) \leq t_0 / s_n$
	\begin{equation} \label{eq:bound-xi+theta}
		\max \left( \E \left( |\xi_n^a|^i \right), \ \E \left( |\theta_n^a|^i \right), \E \left( |\theta_n^b|^i \right) \right) \leq C (c-a) s_n.
	\end{equation}
\end{lemma}

\begin{proof}
	The results for $\xi_n^a$ and $\theta_n^a$ are direct consequences of~\eqref{eq:bound-p-xi} and~\eqref{eq:bound-p-theta} and the finiteness of $C_i$: indeed, using these two results we have for instance for $\xi_n^a$
	\[ \E \left( |\xi_n^a|^i \right) = 1-p_{n,\xi}^a(a,b) + p_{n,\xi}^a(a,b) \E \left( (G_n(b-a))^i \right) \leq C (b-a) s_n + C_i (b-a) s_n \]
	and similarly for $\theta_n^a$.	The result for $\theta_n^b$ is also straightforward because
	\begin{multline*}
		p_{n,\theta}^{b}(a,b,c) = \PX_n^{b} \left( T(c) < T(b) \, | \, T(b) < T(a) < T(0) \right)\\
		= \PX_n^b\left( T(c) < T(b) \, | \, T(b) < T(a) \right) = p_{n,\xi}^{b-a}(c-a, b-a).
	\end{multline*}
	and so the result follows similarly as for $\xi_n^a$.
\end{proof}

Recall the random variables $\widetilde \xi_n^a$ and $\widetilde \theta_n^a$ defined in~\eqref{eq:def-tilde}.

\begin{lemma}
	For any $A > a_0$, there exists a finite constant $C$ such that for all $n \geq 1$ and all $a_0 \leq a < b < c \leq A$ with $(b-a) \vee (c-b) \leq t_0 / s_n$,
	\begin{equation} \label{eq:bound-xi+theta^*}
		\E \left( \left| \widetilde \xi_n^a \right|^i \right) \leq C (b-a) s_n \ \text{ and } \ \E \left( \left| \widetilde \theta_n^a \right|^i \right) \leq C (c-a) s_n.
	\end{equation}
\end{lemma}

\begin{proof}
	Combining the two definitions~\eqref{eq:def-xi} and~\eqref{eq:def-tilde}, we obtain
	\begin{align*}
		\E \left( \left| \widetilde \xi_n^a \right|^i \right) & = \int_0^\infty \left( 1-p_{n,\xi}^{x_0}(a,b) + p_{n,\xi}^{x_0}(a,b) \E \left( (G_n(b-a))^i \right) \right) \P(\chi_n^a \in dx_0)\\
		& \leq C (b-a) s_n + \P(a \leq \chi_n^a \leq b) + C_i (b-a) s_n
	\end{align*}
	using~\eqref{eq:bound-p-xi} to obtain the inequality. We obtain similarly for $\widetilde \theta_n^a$, using~\eqref{eq:bound-p-theta} instead of~\eqref{eq:bound-p-xi},
	\[ \E \left( \left| \widetilde \theta_n^a \right|^i \right) \leq C (c-b) s_n + \P(b \leq \chi_n^a \leq c) + C_i (c-b) s_n. \]
	
	Thus the result will be proved if we can show that $\P(a \leq \chi_n^a \leq c) \leq C (c-a) s_n$; remember that $\chi_n^a$ is by definition equal in distribution to $X(0)$ under $\PX_n^*(\, \cdot \, | \, T(a) < T(0))$. Because $X$ is spectrally positive and $a \leq A$, it holds that
	\[
		\PX_n^* \left( a < X(0) \leq c \, | \, T(a) < T(0) \right) = \frac{\PX_n^* \left( a < X(0) \leq c \right)}{\PX_n^* \left( T(a) < T(0) \right)} \leq \frac{\PX_n^* \left( a < X(0) \leq c \right)}{\PX_n^* \left( T(A) < T(0) \right)}.
	\]
	
	Since $\PX_n^* \left( T(A) < T(0) \right) = \PX_n^0 \left( T(A) < T(0) \right)$ by Lemma~\ref{lemma:overshoot}, Lemma~\ref{lemma:conv-hitting-times} implies that
	\[ \inf_{n \geq 1} \left( r_n \PX_n^* \left( T(A) < T(0) \right) \right) > 0 \]
	which leads to $\PX_n^* \left( a < X(0) \leq c \, | \, T(a) < T(0) \right) \leq C r_n \PX_n^* \left( a < X(0) \leq c \right)$. We have by definition
	\[
		r_n \PX_n^* \left( a < X(0) \leq c \right) = (\alpha-1) r_n \int_{a s_n}^{c s_n} \frac{du}{(1+u)^\alpha} \leq C (c-a) \leq C (c-a) s_n
	\]
	which achieves the proof.
\end{proof}

To control the sum of i.i.d.\ random variables, we will repeatedly use the following simple combinatorial lemma. In the sequel for $I \in \N$ and $\beta \in \N^I$ note $|\beta| = \sum \beta_i$ and $\norm{\beta} = \sum i \beta_i$.

\begin{lemma} \label{lemma:sum}
	Let $(Y_{k})$ be i.i.d.\ random variables with common distribution $Y$. Then for any even integer $I \geq 0$ and any $K \geq 0$,
	\[ \E \left[ \left( \sum_{k=1}^K Y_i \right)^I \right] \leq I^I \sum_{\beta \in \N^I: \norm{\beta} = I} K^{|\beta|} \prod_{i=1}^I \left| \E \left( Y^i \right) \right|^{\beta_i}. \]
\end{lemma}

\begin{proof}
	We have $\E \left[ \left( Y_1 + \cdots + Y_K \right)^I \right] = \sum_{1 \leq k_1, \ldots, k_I \leq K} \E(Y_{k_1} \cdots Y_{k_I})$. Since the $(Y_{k})$'s are i.i.d., we have $\E \left( Y_{k_1} \cdots Y_{k_I} \right) = m_1^{\beta_1} \cdots m_I^{\beta_I}$ with $m_i = \E(Y^i)$ and $\beta_i$ the number of $i$-tuples of $k$, i.e., $\beta_1$ is the number of singletons, $\beta_2$ the number of pairs, etc \ldots\ Since $I$ is even, this leads to
	\[
		\E \left[ \left( Y_{1} + \cdots + Y_K \right)^I \right] \leq \sum_{0 \leq \beta_1, \ldots, \beta_I \leq K: \norm{\beta} = I } A_{I,K}(\beta) \, |m_1|^{\beta_1} \cdots |m_I|^{\beta_I}
	\]
	with $A_{I,K}(\beta)$ the number of $I$-tuples $k \in \{1, \ldots, K\}^I$ with exactly $\beta_i$ $i$-tuples for each $i = 1, \ldots, I$. There are $K (K-1) \ldots (K - (|\beta|-1))$ different ways of choosing the $|\beta|$ different values taken by $k$, thus $A_{I,K}(\beta) = K (K-1) \ldots (K - (|\beta|-1)) \times B(I, |\beta|)$ with $B(i, a)$ the number of ways of assigning $i$ objects into $a$ different boxes in such a way that no box is empty, so that $A_{I,K}(\beta) \leq K^{|\beta|} I^{|\beta|} \leq K^{|\beta|} I^I$ since $|\beta| \leq \norm{\beta} = I$.
\end{proof}

In the sequel, we will use the inequality
\[ \E \left( (G_n(a))^i \right) \leq i! (r_n W_n(a))^i, \ i \geq 1, a > 0, \]
which comes from the fact that $G_n(a)$ is stochastically dominated by an exponential random variable with parameter $1-p_n^a(a) = 1/(r_n W_n(a))$. We now use the previous bounds on the moments to control the probability
\[ \P \left( \left| \widetilde \xi_n^a + \sum_{k=1}^{G_n(a)} \xi_{n,k}^a \right| \geq \lambda r_n \right). \]
We will see that it achieves a linear bound (in $b-a$) which justifies the need of the min later on.

\begin{lemma} \label{lemma:control-xi}
	For any $A > a_0$ and any even integer $I \geq 2$, there exists a finite constant $C$ such that for all $n \geq 1$, all $\lambda > 0$ and all $a_0 \leq a < b \leq A$ with $b-a \leq t_0 / s_n$,
	\[
		\P \left( \left| Y + \sum_{k=1}^{G_n(a)} \xi_{n,k}^a \right| \geq \lambda r_n \right)	\leq  C \frac{b-a}{\lambda^I} s_n r_n^{-I/2}
	\]
	where $Y$ is any random variable equal in distribution either to $\widetilde \xi_n^a$ or to $G_n(b-a)$.
\end{lemma}

\begin{proof}
	Using first the triangular inequality and then Markov inequality gives
	\[ \P \left( \left| Y + \sum_{k=1}^{G_n(a)} \xi_{n,k}^a \right| \geq \lambda r_n \right) \leq (2/\lambda r_n)^{I} \left( \E \left( Y^I \right) + \E \left[ \left( \sum_{k=1}^{G_n(a)} \xi_{n,k}^a \right)^I \right] \right). \]
	Using the independence between $G_n(a)$ and $(\xi_{n,k}^a, k \geq 1)$ together with Lemma~\ref{lemma:sum} gives
	\begin{align*}
		\E \left[ \left( \sum_{k=1}^{G_n(a)} \xi_{n,k}^a \right)^I \right] & \leq C \sum_{\beta \in \N^I: \norm{\beta} = I} \E\left((G_n(a))^{|\beta|}\right) \prod_{i=1}^I \left| \E \left( (\xi_n^a)^i \right) \right|^{\beta_i}\\
		& \leq C \sum_{\beta \in \N^I: \norm{\beta} = I} (r_n W_n(a))^{|\beta|} \prod_{i=1}^I \left| \E \left( (\xi_n^a)^i \right) \right|^{\beta_i}.
	\end{align*}
	Lemma~\ref{lemma:bound-xi} gives the bound
	\begin{multline*}
		\prod_{i=1}^I \left| \E \left( (\xi_n^a)^i \right) \right|^{\beta_i} \leq C \left( (b-a) s_n / (r_n W_n(a)) \right)^{\beta_1} \prod_{i=2}^I ((b-a)s_n)^{\beta_i}\\
		= C ((b-a) s_n)^{|\beta|} \left( r_n W_n(a) \right)^{-\beta_1} \leq C (b-a) s_n \left( r_n W_n(a) \right)^{-\beta_1}
	\end{multline*}
	where $((b-a) s_n)^{|\beta|} \leq C (b-a) s_n$ follows from the fact that $(b-a) s_n \leq t_0$ while $|\beta| \geq 1$. Using~\eqref{eq:bound-xi+theta^*} for the case $Y = \widetilde \xi_n^a$ and the finiteness of $C_I$ for the case $Y = G_n(b-a)$, one can write $\E(Y^I) \leq C (b-a) s_n$, which gives
	\[
		\P \left( \left| Y + \sum_{k=1}^{G_n(a)} \xi_{n,k}^a \right| \geq \lambda r_n \right) \leq C (\lambda r_n)^{-I} (b-a) s_n \left( 1 +  \sum_{\beta \in \N^I: \norm{\beta} = I} (r_n W_n(a))^{|\beta| - \beta_1} \right).
	\]
	But
	\begin{equation} \label{eq:beta}
		|\beta| - \beta_1 = \sum_{i=2}^I \beta_i \leq \sum_{i=2}^I (i/2) \beta_i = \frac{1}{2} \left( \norm \beta - \beta_1 \right) \leq I/2
	\end{equation}
	and $r_n W_n(a) \geq 1$ so $(r_n W_n(a))^{|\beta| - \beta_1} \leq (r_n W_n(a))^{I/2}$, which proves the result.
\end{proof}

\begin{lemma} \label{lemma:bound-N}
	For any $i \geq 1$ and $A > 0$, it holds that
	\[ \sup \left\{ r_n^{-i} \E \left((N_{n,a})^i \right) : n \geq 1, 0 < a < b \leq A \right\} < +\infty \]
	and
	\[ \sup \left\{ r_n^{-i} \E \left((\widetilde N_{n,b}^{a})^i \right) : n \geq 1, 0 < a < b < c \leq A, b-a \leq t_0 / s_n \right\} < +\infty. \]
\end{lemma}

\begin{proof}
	The result on $N_{n,a}$ comes from the following inequality $\E((N_{n,a})^i) \leq \E((G_n(a))^i)$. For $\widetilde N_{n,b}^{a}$, we use the fact that $|\widetilde \xi_n^a|$ is stochastically dominated by $1 + G_n(b-a)$ (since for any $x_0 > 0$, $|\xi_n^{x_0}|$ is), thus
	\[ \E \left( ( \widetilde N_{n,b}^{a})^i \right) \leq \E \left( \left( \sum_{k=1}^{G_n(a)+1} (1+G_{n,k}(b-a)) \right)^i \right) \]
	with $(G_{n,k}(b-a), k \geq 1)$ i.i.d.\ with common distribution $G_n(b-a)$, independent of $G_n(a)$. Thus Lemma~\ref{lemma:sum} gives
	\[ \E \left( ( \widetilde N_{n,b}^{a})^i \right) \leq C \sum_{\beta \in \N^i: \norm \beta = i} \E \left( (1+G_n(a))^{|\beta|} \right) \prod_{k=1}^i \left[ \E \left( (1+G_{n}(b-a))^k \right) \right]^{\beta_k}. \]
	Since $G_n(a)$ is stochastically dominated by an exponential random variable with parameter $1-p_n^a(a) = 1/(r_n W_n(a))$ and $G_n(b-a)$ is integer valued, so that $(1+G_n(b-a))^k \leq (1+G_n(b-a))^i$ for any $1 \leq k \leq i$, we get, using that $|\beta| \leq i$ and that all quantities are greater than $1$,
	\[ \E \left( ( \widetilde N_{n,b}^{a})^i \right) \leq C E \left( (1+E r_n W_n(a))^{i} \right) \left[ \E \left( (1+G_{n}(b-a))^i \right) \right]^{i} \]
	where $E$ is a mean-$1$ exponential random variable. Using that for each $1 \leq k \leq i$
	\[ \E\left((G_n(b-a))^k\right) \leq \E\left((G_n(b-a))^i\right) \leq C_i (b-a) s_n \leq C_i t_0, \]
	one gets
	\[ \sup \left\{ \left[ \E \left( (1+G_{n}(b-a))^i \right) \right]^{i} : n \geq 1, b-a \leq t_0 / s_n \right\} < +\infty. \]
	Together with the inequality
	\[ \E \left( (1+E r_n W_n(a))^{i} \right) \leq \E \left( (1+E r_n W_n(A))^{i} \right) \leq C r_n^{i} \]
	this concludes the proof.
\end{proof}

We can now prove Proposition~\ref{prop:case<}. Remember that we must find constants $C$ and $\gamma > 0$ such that
\[
	\PX_n^* \left( \left| L^0(c) - L^0(b) \right| \wedge \left| L^0(b) - L^0(a) \right| \geq \lambda \, | \, T(a) < T(0) \right) \leq C \frac{(c-a)^{3/2}}{\lambda^\gamma}
\]
uniformly in $n \geq 1$, $\lambda > 0$ and $a_0 \leq a < b < c \leq A$ with $(b-a) \vee (c-b) \leq t_0 / s_n$.

\begin{proof}[Proof of Proposition~\ref{prop:case<}]
	Fix four even integers $I_1, I_2, I_3, I_4$. By~\eqref{eq:fd-integrated},
	\begin{multline*}
		\PX_n^* \left( \left| L^0(c) - L^0(b) \right| \wedge \left| L^0(b) - L^0(a) \right| \geq \lambda \, | \, T(a) < T(0) \right)\\
		= \P \left( \left| \widetilde \xi_n^a + \sum_{k=1}^{G_n(a)} \xi_{n,k}^a \right| \wedge \left| \widetilde \theta_n^a \indicator{\widetilde \xi_n^a \geq 0} + \sum_{k=1}^{N_{n,a}} \theta_{n,k}^a + \sum_{k=1}^{\widetilde N_{n,b}^{a}} \theta_{n,k}^b \right| \geq \lambda_n \right)
	\end{multline*}
	with $\lambda_n = \lambda r_n$. Let $\Fcal$ be the $\sigma$-algebra generated by $\chi_n^a$, $G_n(a)$, $\widetilde \xi_n^a$ and the $(\xi_{n,k}^a, k \geq 1)$. Then the above probability is equal to
	\[
		\E \left\{ \pi \, ; \, \left| \widetilde \xi_n^a + \sum_{k=1}^{G_n(a)} \xi_{n,k}^a \right| \geq \lambda_n \right\}
	\]
	with $\pi$ the random variable
	\[ \pi = \P \left( \left| \widetilde \theta_n^a \indicator{\widetilde \xi_n^a \geq 0} + \sum_{k=1}^{N_{n,a}} \theta_{n,k}^a + \sum_{k=1}^{\widetilde N_{n,b}^{a}} \theta_{n,k}^b \right| \geq \lambda_n \, \Big | \, \Fcal \right) \leq \widetilde \pi + \pi_a + \pi_b \]
	where
	\[ \widetilde \pi = \P \left( \left| \widetilde \theta_n^a \indicator{\widetilde \xi_n^a \geq 0} \right| \geq \lambda_n / 3 \, \Big | \, \Fcal \right) =  \indicator{\widetilde \xi_n^a \geq 0} \P \left( \left| \widetilde \theta_n^a \right| \geq \lambda_n / 3 \, \Big | \, \chi_n^a \right), \]
	\[ \pi_a = \P \left( \left| \sum_{k=1}^{N_{n,a}} \theta_{n,k}^a \right| \geq \lambda_n / 3 \, \Big | \, N_{n,a} \right) \text{ and } \pi_b = \P \left( \left| \sum_{k=1}^{\widetilde N_{n,b}^{a}} \theta_{n,k}^b \right| \geq \lambda_n / 3 \, \Big | \, \widetilde N_{n,b}^{a} \right). \]

	The two terms $\pi_a$ and $\pi_b$ can be dealt with very similarly. Fix $u = a$ or $b$, and denote by $N_u$ the random variable $N_{n,a}$ if $u = a$ or $\widetilde N_{n,b}^{a}$ if $u = b$. With this notation, $(\theta_{n,k}^u, k \geq 1)$ are i.i.d.\ and independent from $N_u$, so that Markov inequality and Lemma~\ref{lemma:sum} give
	\[
		\pi_u \leq (3I_1 / \lambda_n)^{I_1} \sum_{\beta \in \N^{I_1}: \norm{\beta} = I_1} N_u^{|\beta|} \prod_{i=1}^{I_1} \left| \E \left( \left( \theta_{n}^u \right)^i \right) \right|^{\beta_i}.
	\]
	By~\eqref{eq:bound-xi+theta},
	\[ \prod_{i=1}^{I_1} \left| \E \left( \left( \theta_n^u \right)^i \right) \right|^{\beta_i} \leq C ((c-a) s_n)^{|\beta|} \leq C (c-a) s_n \]
	since $1 \leq |\beta| \leq I_1$ and $(c-a) s_n \leq t_0$. Since $N_u$ is integer valued it holds that $N_u^{|\beta|} \leq N_u^{I_1}$ and finally this gives
	\[ \pi_u \leq C \lambda^{-I_1} (N_u / r_n)^{I_1} (c-a) s_n. \]
	Applying Cauchy-Schwarz inequality yields
	\[ \E \left\{ \pi_u \, ; \, \left| \widetilde \xi_n^a + \sum_{k=1}^{G_n(a)} \xi_{n,k}^a \right| \geq \lambda_n \right\} \leq C \lambda^{-I_1} (c-a) s_n \sqrt{ \E \left((N_u / r_n)^{2I_1} \right) \P \left( \left| \widetilde \xi_n^a + \sum_{k=1}^{G_n(a)} \xi_{n,k}^a \right| \geq \lambda_n \right) } \]
	and finally, Lemma~\ref{lemma:control-xi} with $Y = \widetilde \xi_n^a$ gives, together with Lemma~\ref{lemma:bound-N},
	\[
		\E \left\{ \pi_u \, ; \, \left| \widetilde \xi_n^a + \sum_{k=1}^{G_n(a)} \xi_{n,k}^a \right| \geq \lambda_n \right\} \leq C \frac{(c-a)^{3/2}}{\lambda^{I_1+I_2/2}} s_n^{3/2} r_n^{-I_2/4}.
	\]
	It remains to control the term $\widetilde \pi$: in $\{ \widetilde \xi_n^a \geq 0 \}$, $\widetilde \xi_n^a$ is equal in distribution to $G_n(b-a)$ and is independent of everything else, thus we have
	\begin{align*}
		\E \left\{ \widetilde \pi \, ; \, \left| \widetilde \xi_n^a + \sum_{k=1}^{G_n(a)} \xi_{n,k}^a \right| \geq \lambda_n \right\} & = \E \left\{ \P \left( \left| \widetilde \theta_n^a \right| \geq \lambda_n / 3 \, \big | \, \chi_n^a \right) \, ; \, \left| \widetilde \xi_n^a + \sum_{k=1}^{G_n(a)} \xi_{n,k}^a \right| \geq \lambda_n , \widetilde \xi_n^a \geq 0 \right\}\\
		& \leq C \lambda_n^{-I_3} \E \left\{ \E \left( \left| \widetilde \theta_n^a \right|^{I_3} \, \big | \, \chi_n^a \right) \, ; \, \left| G_n(b-a) + \sum_{k=1}^{G_n(a)} \xi_{n,k}^a \right| \geq \lambda_n \right\}.
	\end{align*}
	
	Since $\E( | \widetilde \theta_n^a|^{I_3} \, | \, \chi_n^a )$ is independent of $G_n(b-a) + \sum_{k=1}^{G_n(a)} \xi_{n,k}^a$, we get
	\begin{align*}
		\E \left\{ \widetilde \pi \, ; \, \left| \widetilde \xi_n^a + \sum_{k=1}^{G_n(a)} \xi_{n,k}^a \right| \geq \lambda_n \right\} & \leq C \lambda_n^{-I_3} \E \left( \left| \widetilde \theta_n^a \right|^{I_3} \right) \P \left( \left| G_n(b-a) + \sum_{k=1}^{G_n(a)} \xi_{n,k}^a \right| \geq \lambda_n \right)\\
		& \leq C \lambda_n^{-I_3-I_4} (c-a)^2 s_n^2 r_n^{-I_4/2}
	\end{align*}
	where the second inequality follows using~\eqref{eq:bound-xi+theta^*} and Lemma~\ref{lemma:control-xi} with $Y = G_n(b-a)$. Since $(c-a) s_n \leq t_0$, we have $((c-a)s_n)^2 \leq C ((c-a)s_n)^{3/2}$ and finally, gathering the previous inequalities, one sees that we have derived the bound
	\begin{multline*}
		\PX_n^* \left( \left| L^0(c) - L^0(b) \right| \wedge \left| L^0(b) - L^0(a) \right| \geq \lambda \, | \, T(a) < T(0) \right)\\
		\leq C (c-a)^{3/2} s_n^{3/2} \left( \lambda^{-I_1 - I_2/2} r_n^{-I_2/4} + \lambda^{-I_3 - I_4} r_n^{-I_4/2} \right).
	\end{multline*}
	Now choose $I_2$ and $I_4$ large enough such that both sequences $(s_n^{3/2} r_n^{-I_2/4})$ and $(s_n^{3/2} r_n^{-I_4/2})$ are bounded: this is possible since for any $\beta \in \R$, $s_n r_n^{-\beta} = s_n^{1-\beta(\alpha-1)}$. Moreover, choose $I_2$ not only even but a multiple of $4$. Then once $I_2$ and $I_4$ are fixed, choosing $I_1$ and $I_3$ in such a way that $I_1 + I_2 / 2 = I_3 + I_4$ concludes the proof.
\end{proof}

\subsection{Case $b-a \geq t_0 / s_n$}

We now consider the simpler case $b-a \geq t_0 / s_n$ and prove Proposition~\ref{prop:case>}.

\begin{lemma} \label{lemma:bound-xi^*-case>}
	For any $i \geq 1$, there exists a finite constant $C$ such that for all $n \geq 1$ and all $0 < a < b$ such that $b-a \geq t_0 / s_n$,
	\[ \E \left( |\widetilde \xi_n^a|^i \right) \leq C (r_n W_n(b-a))^i. \]
\end{lemma}

\begin{proof}
	In view of~\eqref{eq:def-tilde}, it is enough to show that $\E \left( |\xi_n^{x_0}|^i \right) \leq C (r_n W_n(b-a))^i$ for every $x_0 > 0$. Since $b-a \geq t_0 / s_n$, exploiting the monotonicity of $w_n$ gives
	\[ p_n^{b-a}(b-a) = 1 - \frac{w_n(0)}{w_n((b-a)s_n)} \geq 1 - \frac{1}{w_n(t_0)} \geq \frac{1}{2} \]
	since $t_0$ has been chosen such that $w_n(t_0) \geq 2$. Since $G_n(b-a)$ is a geometric random variable with parameter $p_n^{b-a}(b-a)$, we have
	\[ \E\left( (G_n(b-a))^i \right) \geq \E(G_n(b-a)) = \frac{p_n^{b-a}(b-a)}{1-p_n^{b-a}(b-a)} \geq 1, \]
	using $p_n^{b-a}(b-a) \geq 1/2$. Thus for any $x_0 > 0$,
	\begin{align*}
		\E \left( |\xi_n^{x_0}|^i \right) & = 1-p_{n,\xi}^{x_0}(a,b) + p_{n,\xi}^{x_0}(a,b) \E\left( (G_n(b-a))^i \right)\\
		& \leq (1-p_{n,\xi}^{x_0}(a,b)) \E\left( (G_n(b-a))^i \right) + p_{n,\xi}^{x_0}(a,b) \E\left( (G_n(b-a))^i \right).
	\end{align*}
	
	This last quantity is equal to $\E\left((G_n(b-a))^i\right)$ and so the inequality $\E\left((G_n(b-a))^i\right) \leq i! (r_n W_n(b-a))^i$ achieves the proof.
\end{proof}

\begin{lemma} \label{lemma:bound-xi-case>}
	For any $i \geq 1$, there exists a finite constant $C$ such that for all $n \geq 1$ and all $0 < a < b$ with $b-a \geq t_0 / s_n$,
	\[ \E \left( \left| \xi_n^a \right|^i \right) \leq C (r_n W_n(b-a))^{i-1}. \]
	Moreover, for any $n \geq 1$ and $0 < a < b$,
	\[ \left| \E(\xi_n^a) \right| \leq \frac{W_n(b-a)}{W_n(a) - W_n(0)} \]
\end{lemma}

\begin{proof}
	By definition~\eqref{eq:def-xi} of $\xi_n$ we have $\E \left( \left| \xi_n^a \right|^i \right) = 1-p_{n,\xi}^a(a,b) + p_{n,\xi}^a(a,b) \E \left( (G_n(b-a))^i \right)$ and so plugging in~\eqref{eq:formula-p-xi-a} gives
	\[ \E \left( \left| \xi_n^a \right|^i \right) \leq 1 + i! \left( \frac{W_n(b-a)}{W_n(0)} \right)^i \frac{W_n(0)(W_n(b) - W_n(b-a))}{W_n(b-a) (W_n(a) - W_n(0))} \leq 2 i! (r_n W_n(b-a))^{i-1} \]
	using $W_n(b) - W_n(b-a) \leq W_n(a) - W_n(0)$ and $1 \leq i! (r_n W_n(b-a))^{i-1}$. The second inequality is a direct consequence of~\eqref{eq:formula-xi} which can be expanded to
	\[ \left| \E \left( \xi_n^a \right) \right| = \frac{W_n(b-a) - W_n(0) - (W_n(b) - W_n(a))}{W_n(a) - W_n(0)}. \]
	The result is proved.
\end{proof}

\begin{proof}[Proof of Proposition~\ref{prop:case>}]
	By~\eqref{eq:fd-integrated} we have
	\[
		\PX_n^* \left( \left| L^0(b) - L^0(a) \right| \geq \lambda \, | \, T(a) < T(0) \right) = \P\left( \left| \widetilde \xi_n^a + \sum_{k=1}^{G_n(a)} \xi_{n,k}^a\right| \geq \lambda r_n \right).
	\]
	We have
	\begin{multline*}
		\P\left( \left| \widetilde \xi_n^a + \sum_{k=1}^{G_n(a)} \xi_{n,k}^a\right| \geq \lambda r_n \right)\\
		\leq C(\lambda r_n)^{-I} \left( \E \left( |\widetilde \xi_n^a|^I \right) + \sum_{\beta \in \N^I: \norm \beta = I} \E\left( (G_n(a))^{|\beta|} \right) \prod_{i=1}^I \left| \E \left( (\xi_n^a)^i \right) \right|^{\beta_i} \right)\\
		\leq C\lambda^{-I} \left( (W_n(b-a))^I + \sum_{\beta \in \N^I: \norm \beta = I} r_n^{|\beta| - I} (W_n(a))^{|\beta|} \prod_{i=1}^I \left| \E \left( (\xi_n^a)^i \right) \right|^{\beta_i} \right)
	\end{multline*}
	where the first inequality comes from the triangular inequality, Markov inequality and Lemma~\ref{lemma:sum}, and the second inequality is a consequence of Lemma~\ref{lemma:bound-xi^*-case>} and the fact that $G_n(a)$ is stochastically dominated by an exponential random variable with parameter $1-p_n^a(a)$. Using Lemma~\ref{lemma:bound-xi-case>} and the identity $\sum_{i=2}^I (i-1) \beta_i = I - |\beta|$ gives
	\begin{multline*}
		r_n^{|\beta|-I} (W_n(a))^{|\beta|} \prod_{i=1}^I \left| \E \left( (\xi_n^a)^i \right) \right|^{\beta_i}\\
		\leq C r_n^{|\beta|-I} (W_n(a))^{|\beta|} \left( \frac{W_n(b-a)}{W_n(a) - W_n(0)} \right)^{\beta_1} (r_n W_n(b-a))^{I-|\beta|}
		\\ \leq C (W_n(b-a))^{I + \beta_1-|\beta|}.
	\end{multline*}
	
	Thus
	\[
		\P\left( \left| \widetilde \xi_n^a + \sum_{k=1}^{G_n(a)} \xi_{n,k}^a\right| \geq \lambda r_n \right) \leq C\lambda^{-I} \left( (W_n(b-a))^I + \sum_{\beta \in \N^I: \norm \beta = I} (W_n(b-a))^{I-|\beta|+\beta_1} \right).
	\]
	Since $W_n(t) = w_n(t s_n) / s_n^{\alpha-1}$, it holds that
	\[
		\sup \left\{ \frac{W_n(t)}{t^{\alpha-1}} : n \geq 1, t \geq t_0/s_n \right\} = \sup \left\{ \frac{w_n(t)}{t^{\alpha-1}} : n \geq 1,  t \geq t_0 \right\}
	\]
	which has been shown to be finite in the proof of Lemma~\ref{lemma:prop-w}. Hence the last upper bound yields
	\[
		\P\left( \left| \widetilde \xi_n^a + \sum_{k=1}^{G_n(a)} \xi_{n,k}^a\right| \geq \lambda r_n \right) \leq C\lambda^{-I} \left( (b-a)^{I(\alpha-1)} + \sum_{\beta \in \N^I: \norm \beta = I} (b-a)^{(I-|\beta|+\beta_1)(\alpha-1)} \right).
	\]
	By~\eqref{eq:beta}, $I-|\beta| + \beta_1 \geq I / 2$ and since we consider $b-a \leq A$ this gives
	\[ (b-a)^{(I-|\beta|+\beta_1)(\alpha-1)} \leq C (b-a)^{(\alpha-1) I / 2} \]
	and we finally get the desired bound for $I$ large enough, i.e., such that $I(\alpha-1) \geq 3$. Inspecting the proof of Proposition~\ref{prop:case<} one can check that one can choose the two constants $\gamma$ to be equal.
\end{proof}


\begin{thebibliography}{10}

\bibitem{Barlow88:0}
M.~T. Barlow.
\newblock Necessary and sufficient conditions for the continuity of local time
  of {L}\'evy processes.
\newblock {\em Ann. Probab.}, 16(4):1389--1427, 1988.

\bibitem{Bertoin96:0}
Jean Bertoin.
\newblock {\em L\'evy processes}, volume 121 of {\em Cambridge Tracts in
  Mathematics}.
\newblock Cambridge University Press, Cambridge, 1996.

\bibitem{Bertoin97:1}
Jean Bertoin.
\newblock Exponential decay and ergodicity of completely asymmetric {L}\'evy
  processes in a finite interval.
\newblock {\em Ann. Appl. Probab.}, 7(1):156--169, 1997.

\bibitem{Billingsley99:0}
Patrick Billingsley.
\newblock {\em Convergence of probability measures}.
\newblock Wiley Series in Probability and Statistics: Probability and
  Statistics. John Wiley \& Sons Inc., New York, second edition, 1999.

\bibitem{Borodin81:0}
A.~N. Borodin.
\newblock The asymptotic behavior of local times of recurrent random walks with
  finite variance.
\newblock {\em Teor. Veroyatnost. i Primenen.}, 26(4):769--783, 1981.

\bibitem{Borodin84:0}
A.~N. Borodin.
\newblock Asymptotic behavior of local times of recurrent random walks with
  infinite variance.
\newblock {\em Teor. Veroyatnost. i Primenen.}, 29(2):312--326, 1984.

\bibitem{Borodin86:0}
A.~N. Borodin.
\newblock On the character of convergence to {B}rownian local time. {I}, {II}.
\newblock {\em Probab. Theory Relat. Fields}, 72(2):231--250, 251--277, 1986.

\bibitem{Caballero09:0}
Ma.~Emilia Caballero, Amaury Lambert, and Ger{\'o}nimo Uribe~Bravo.
\newblock Proof(s) of the {L}amperti representation of continuous-state
  branching processes.
\newblock {\em Probab. Surv.}, 6:62--89, 2009.

\bibitem{Chan10:0}
T.~Chan, A.~Kyprianou, and M.~Savov.
\newblock Smoothness of scale functions for spectrally negative {L}\'evy
  processes.
\newblock {\em Probability Theory and Related Fields}, pages 1--18, 2010.
\newblock 10.1007/s00440-010-0289-4.

\bibitem{Csaki83:0}
E.~Cs{\'a}ki and P.~R{\'e}v{\'e}sz.
\newblock Strong invariance for local times.
\newblock {\em Z. Wahrsch. Verw. Gebiete}, 62(2):263--278, 1983.

\bibitem{Csorgo85:0}
M.~Cs{\"o}rg{\H{o}} and P.~R{\'e}v{\'e}sz.
\newblock On strong invariance for local time of partial sums.
\newblock {\em Stochastic Process. Appl.}, 20(1):59--84, 1985.

\bibitem{Duquesne02:0}
Thomas Duquesne and Jean-Fran{\c{c}}ois Le~Gall.
\newblock Random trees, {L}\'evy processes and spatial branching processes.
\newblock {\em Ast\'erisque}, (281):vi+147, 2002.

\bibitem{Eisenbaum93:0}
Nathalie Eisenbaum and Haya Kaspi.
\newblock A necessary and sufficient condition for the {M}arkov property of the
  local time process.
\newblock {\em Ann. Probab.}, 21(3):1591--1598, 1993.

\bibitem{Feller71:0}
William Feller.
\newblock {\em An introduction to probability theory and its applications.
  {V}ol. {II}.}
\newblock Second edition. John Wiley \& Sons Inc., New York, 1971.

\bibitem{Grimvall74:0}
Anders Grimvall.
\newblock On the convergence of sequences of branching processes.
\newblock {\em The Annals of Probability}, 2(6):1027--1045, 1974.

\bibitem{Haccou07:0}
Patsy Haccou, Peter Jagers, and Vladimir~A. Vatutin.
\newblock {\em Branching processes: variation, growth, and extinction of
  populations}.
\newblock Cambridge Studies in Adaptive Dynamics. Cambridge University Press,
  Cambridge, 2007.

\bibitem{Helland78:0}
Inge~S. Helland.
\newblock Continuity of a class of random time transformations.
\newblock {\em Stochastic Processes Appl.}, 7(1):79--99, 1978.

\bibitem{Jacod03:0}
Jean Jacod and Albert~N. Shiryaev.
\newblock {\em Limit theorems for stochastic processes}, volume 288 of {\em
  Grundlehren der Mathematischen Wissenschaften [Fundamental Principles of
  Mathematical Sciences]}.
\newblock Springer-Verlag, Berlin, second edition, 2003.

\bibitem{Jain84:1}
Naresh~C. Jain and William~E. Pruitt.
\newblock An invariance principle for the local time of a recurrent random
  walk.
\newblock {\em Z. Wahrsch. Verw. Gebiete}, 66(1):141--156, 1984.

\bibitem{Kang97:0}
Ju-Sung Kang and In-Suk Wee.
\newblock A note on the weak invariance principle for local times.
\newblock {\em Statist. Probab. Lett.}, 32(2):147--159, 1997.

\bibitem{Kella05:0}
Offer Kella, Bert Zwart, and Onno Boxma.
\newblock Some time-dependent properties of symmetric {$M/G/1$} queues.
\newblock {\em J. Appl. Probab.}, 42(1):223--234, 2005.

\bibitem{Kesten65:0}
Harry Kesten.
\newblock An iterated logarithm law for local time.
\newblock {\em Duke Math. J.}, 32:447--456, 1965.

\bibitem{Khoshnevisan93:0}
Davar Khoshnevisan.
\newblock An embedding of compensated compound {P}oisson processes with
  applications to local times.
\newblock {\em Ann. Probab.}, 21(1):340--361, 1993.

\bibitem{Knight63:0}
F.~B. Knight.
\newblock Random walks and a sojourn density process of {B}rownian motion.
\newblock {\em Trans. Amer. Math. Soc.}, 109:56--86, 1963.

\bibitem{Kuznetsov11:0}
Alexey Kuznetsov, Andreas~E. Kyprianou, and V{\'\i}ctor Rivero.
\newblock The theory of scale functions for spectrally negative {L}\'evy
  processes.
\newblock Available on arXiv, 2011.

\bibitem{Kyprianou10:0}
Andreas Kyprianou, V{\'\i}ctor Rivero, and Renming Song.
\newblock Convexity and {S}moothness of {S}cale {F}unctions and de {F}inetti's
  {C}ontrol {P}roblem.
\newblock {\em Journal of Theoretical Probability}, 23:547--564, 2010.
\newblock 10.1007/s10959-009-0220-z.

\bibitem{Lambert10:0}
Amaury Lambert.
\newblock The contour of splitting trees is a {L}\'evy process.
\newblock {\em Ann. Probab.}, 38(1):348--395, 2010.

\bibitem{Lambert12:0}
Amaury Lambert and Florian Simatos.
\newblock The weak convergence of regenerative processes using some excursion
  path decompositions.
\newblock Preprint available on arXiv, 2012.

\bibitem{Lambert11:0}
Amaury Lambert, Florian Simatos, and Bert Zwart.
\newblock Scaling limits via excursion theory: {I}nterplay between
  {Crump-Mode-Jagers branching processes and Processor-Sharing queues}.
\newblock Submitted, available on arXiv, 2011.

\bibitem{Lamperti67:1}
John Lamperti.
\newblock Continuous-state branching processes.
\newblock {\em Bulletin of the American Mathematical Society}, 73:382--386,
  1967.

\bibitem{Lamperti67:0}
John Lamperti.
\newblock The {Limit of a Sequence of Branching Processes}.
\newblock {\em Probability {T}heory and {R}elated {F}ields}, 7(4):271--288,
  August 1967.

\bibitem{Limic01:0}
Vlada Limic.
\newblock A {LIFO} queue in heavy traffic.
\newblock {\em Ann. Appl. Probab.}, 11(2):301--331, 2001.

\bibitem{Perkins82:0}
Edwin Perkins.
\newblock Weak invariance principles for local time.
\newblock {\em Z. Wahrsch. Verw. Gebiete}, 60(4):437--451, 1982.

\bibitem{Revesz81:0}
P.~R{\'e}v{\'e}sz.
\newblock Local time and invariance.
\newblock In {\em Analytical methods in probability theory ({O}berwolfach,
  1980)}, volume 861 of {\em Lecture Notes in Math.}, pages 128--145. Springer,
  Berlin, 1981.

\bibitem{Revesz81:1}
P.~R{\'e}v{\'e}sz.
\newblock A strong invariance principle of the local time of {RV}s with
  continuous distribution.
\newblock {\em Studia Sci. Math. Hungar.}, 16(1-2):219--228, 1981.

\bibitem{Robert03:0}
Philippe Robert.
\newblock {\em Stochastic Networks and Queues}.
\newblock Stochastic Modelling and Applied Probability Series. Springer-Verlag,
  New York, 2003.
\newblock xvii+398 pp.

\bibitem{Sagitov94:0}
S.~M. Sagitov.
\newblock General branching processes: convergence to {I}rzhina processes.
\newblock {\em J. Math. Sci.}, 69(4):1199--1206, 1994.
\newblock Stability problems for stochastic models (Kirillov, 1989).

\bibitem{Sagitov95:0}
Serik Sagitov.
\newblock A key limit theorem for critical branching processes.
\newblock {\em Stochastic Process. Appl.}, 56(1):87--100, 1995.

\bibitem{Stone63:0}
Charles Stone.
\newblock Limit theorems for random walks, birth and death processes, and
  diffusion processes.
\newblock {\em Illinois J. Math.}, 7:638--660, 1963.

\end{thebibliography}
\end{document}